\documentclass[11pt,reqno]{amsart}

\usepackage[dvipsnames]{xcolor}
\usepackage{amsfonts}
\usepackage{amsmath,amscd}
\usepackage{amssymb}
\usepackage{tikz-cd}
\usepackage{amsthm}
\usepackage{cases}

\usepackage{galois}
\usepackage{chemarrow}
\usepackage[all]{xy}
\usepackage{graphicx}
\usepackage[colorlinks,
            linkcolor=magenta,
            anchorcolor=black,
            citecolor=red
            ]{hyperref}	
\usepackage{mathrsfs}
\usepackage{extarrows}
\usepackage{texnames}
\usepackage{textcase}

\usepackage[deletedmarkup=xout]{changes}
\definechangesauthor[name={Per cusse}, color=OrangeRed]{.}

\makeatletter
\def\section{%
  \@startsection{section}{1}
    {\z@}
    {2.0ex plus 0.8ex minus .1ex}
    {1.0ex plus .2ex}
    {\large\bfseries\boldmath\centering\MakeTextUppercase}%
}
\makeatother

\def\ve{{\rm \varepsilon}}
\newcommand{\dd}{\sqrt{-1}\partial\bar{\partial}}
\newcommand{\krf}{K\"ahler-Ricci flow\ }
\newcommand{\dl}{\left(\frac{\partial}{\partial t}-\Delta_{\omega_{\gamma\varepsilon j}(t)}\right)}
\newcommand{\osc}{\mathop{\mathrm{osc}}}
\newcommand{\tr}{\mathrm{tr}}

\newtheorem{thm}{Theorem}[section]
\newtheorem{lem}[thm]{Lemma}

\newtheorem{rem}[thm]{Remark}
\newtheorem{pro}[thm]{Proposition}
\newtheorem{defi}[thm]{Definition}

\numberwithin{equation}{section}

\usepackage[utf8]{inputenc}

\begin{document}

\title[\sc\small R\MakeLowercase{egularizing property of the twisted conical} K\MakeLowercase{\"ahler-}R\MakeLowercase{icci flow}]{\sc\LARGE R\MakeLowercase{egularizing property of the twisted conical} K\MakeLowercase{\"ahler}-R\MakeLowercase{icci flow}}
\keywords{conical K\"ahler-Ricci flow, twisted K\"ahler-Ricci flow, conical singularity.}
\author[\sc\small J\MakeLowercase{iawei} L\MakeLowercase{iu}, S\MakeLowercase{hiyu} Z\MakeLowercase{hang and} X\MakeLowercase{i} Z\MakeLowercase{hang}]{\sc\large J\MakeLowercase{iawei} L\MakeLowercase{iu}, S\MakeLowercase{hiyu} Z\MakeLowercase{hang and} X\MakeLowercase{i} Z\MakeLowercase{hang}}
\address{Jiawei Liu\\School of Mathematics and Statistics\\ Nanjing University of Science \& Technology\\ Xiaolingwei Street 200\\ Nanjing 210094\\ China.} \email{jiawei.liu@njust.edu.cn}
\address{Shiyu Zhang\\School of Mathematical Sciences\\ University of Science and Technology of China\\ Jinzhai Road 96\\ Hefei 230026\\ China.} 
\email{shiyu123@mail.ustc.edu.cn}
\address{Xi Zhang\\ School of Mathematics and Statistics\\ Nanjing University of Science \& Technology\\ Xiaolingwei Street 200\\ Nanjing 210094\\ China.} 
\email{mathzx@njust.edu.cn}
\subjclass[2020]{53E30,\ 35K96,\ 35K67.}
\thanks{X. Zhang is supported by National Key R and D Program of China 2020YFA0713100, NSFC (Grant Nos.12141104 and 11721101). J.W. Liu is supported by NSFC (Grant No.12371059), Jiangsu Specially-Appointed Professor Program and Fundamental Research Funds for the Central Universities, S.Y. Zhang is supported by NSFC (Grant No.12371062).}

\maketitle
\vskip -3.99ex

\centerline{\noindent\mbox{\rule{3.99cm}{0.5pt}}}

\vskip 5.01ex

\ \ \ \ {\bf Abstract.} In this paper, we show the regularity and uniqueness of the twisted conical K\"ahler-Ricci flow running from a positive closed current with zero Lelong number, which extends the regularizing property of the smooth twisted K\"ahler-Ricci flow, known as Guedj-Zeriahi's existence theorem \cite{VGAZ2} and Di Nezza-Lu's uniqueness theorem \cite{NL2017}, to the conical singularity case.

%\tableofcontents 
%{\color{blue} text}
%\textcolor{OrangeRed}{text}
%\textcolor{CornflowerBlue}{text}
%\deleted[id=.]{text } 
%\replaced[id=.]{new }{old}
%\added[id=.]{text }

%\tableofcontents 
\vspace{7mm}

\section{Introduction}

The K\"ahler-Ricci flow was first introduced by Cao \cite{Cao} to give a parabolic proof of the Calabi-Yau theorem in 1980s, its general existence was studied by Tsuji \cite{Tsuji} and Tian-Zhang \cite{Tianzzhang}. Then Song-Tian \cite{JSGT} considered the regularity property of the K\"ahler-Ricci flow coming from weak initial metric which admits continuous potential. Later, Sz\'ekelyhidi-Tosatti \cite{GSVT} proved such regularity property for the complex Monge-Amp\`ere flow on K\"ahler manifold, and they also obtained a regularity result for the complex Monge-Amp\`ere equation as the application. Nie \cite{Nie} extended this regularity property to the parabolic flow on Hermitian manifold.  As a generalization of the K\"ahler-Ricci flow, the twisted K\"ahler-Ricci flow which adds a smooth twisted form in K\"ahler-Ricci flow was first studied by Collins-Sz\'ekelyhidi \cite{TC} and the first author \cite{JWL}. By using new arguments, Guedj-Zeriahi \cite{VGAZ2} and Di Nezza-Lu \cite{NL2017} gave the best regularity and uniqueness results of this flow starting from weak metric which only admits zero Lelong number, and then T$\operatorname{\hat{o}}$ extended these results to the complex Monge-Amp\`ere flows on K\"ahler  \cite{TaTo1} and Hermitian manifolds  \cite{TaTo2}. Recently, Zhang \cite{KWZ} proved Tian's partial $C^0$-estimate along this flow, and the authors \cite{JWLXZ3} showed the relation between the convergence of  these flows and the greatest Ricci lower bound.

When adding a non-smooth twisted form in K\"ahler-Ricci flow, more precisely, the twisted form is only a current of the integration along a divisor, such flow was called conical K\"ahler-Ricci flow which was first introduced by Jeffres-Mazzeo-Rubinstein \cite{JMR}. Then Chen-Wang \cite{CW,CW1}, the authors \cite{JWLXZ}, Wang \cite{YQW} and Shen \cite{LMSH2, LMSH1} studied the existence and regularity property of this flow that runs from a class of conical K\"ahler metrics. Later, the authors \cite{JWLXZ1} and the first author and Zhang \cite{JWLCJZ} extended the initial metic to the weak one as in \cite{JSGT, GSVT}, and recently Li-Shen-Zheng \cite{LSZ} obtained a maximum solution when the potential of the weak initial metric is bounded on log canonical pairs. In all above results, the weakest assumption is that the initial metric potential is in $L^\infty$-sense.

In this paper, inspired by Guedj-Zeriahi's existence theorem \cite{VGAZ2} and Di Nezza-Lu's uniqueness theorem \cite{NL2017}, we show that the twisted conical K\"ahler-Ricci flow can also run from a positive closed current with zero Lelong number and then improve the regularity. 

 Let $(X, \omega)$ be a compact K\"ahler manifold with complex dimension $n$ and $D$ be a smooth divisor. Assume that $\hat{\omega}\in [\omega]$ is a positive closed current with zero Lelong number and that $\eta$ is a smooth closed $(1,1)$-form. Fix $\gamma\in(0,1]$, the twisted conical K\"ahler-Ricci flow that is considered in this paper takes the following form,
\begin{equation}\label{CK}
\left\{
\begin{aligned}
 &\ \frac{\partial \omega_{\gamma}(t)}{\partial t}=-{\rm Ric}(\omega_{\gamma}(t))+(1-\gamma)[D]+\eta\\
 &\ \ \ \ \ \ \ \ \ \ \ \ \ \ \ \ \ \ \ \ \ \ \ \ \ \ \ \ \ \ \ \ \ \ \ \ \ \ \ \ \ \ \ \ \ \ \ \ \ \ \ \ \ \ \ \ \ \ \ , \\
 &\ \omega_{\gamma}(t)|_{t=0}=\hat{\omega}
\end{aligned}
\right.\tag{$CKRF^{\eta}_{\gamma}$}
\end{equation}
where $[D]$ is the current of integration along $D$. In particular, if $\eta=0$, the twisted conical K\"ahler-Ricci flow \eqref{CK} is the conical K\"ahler-Ricci flow which has also been studied by Edwards \cite{GEDWA1, GEDWA}, Nomura \cite{Nomura}, Zhang \cite{YSZ1,YSZ} and the authors \cite{JWLXZ2}. 

Denote 
\begin{equation}\label{0131001}
T^{\eta}_{\gamma,\max}=\sup\left\{t\geqslant0\Big| [\omega]+t\left(-c_1(X)+(1-\gamma)c_1(L_D)+[\eta]\right)>0\right\},
\end{equation}
where $L_D$ is the line bundle associated with divisor $D$. We show that the twisted conical K\"ahler-Ricci flow \eqref{CK} admits a unique solution on the maximal time interval $[0,T^\eta_{\gamma,\max})$, which is the main results of this paper.
\begin{thm}\label{0131002}
Let $\hat\omega\in[\omega]$ be a positive current with zero Lelong number. For any $\gamma\in(0,1]$ and any smooth closed $(1,1)$-form $\eta$, there exists a unique weak solution to the twisted conical K\"ahler-Ricci flow \eqref{CK} on $(0,T^\eta_{\gamma,\max})$ in the sense of Definition \ref{04.5}.
\end{thm}
\begin{rem}\label{02281001}
If $\gamma=1$, the twisted conical \krf  \eqref{CK} is the smooth twisted K\"ahler-Ricci flow whose existence and uniqueness was proved by Guedj-Zeriahi (Theorem $A$ in \cite{VGAZ2}) and Di Nezza-Lu (Theorem $1.3$ in \cite{NL2017}).
\end{rem}
The method we use here is smooth approximation which is first used by the authors \cite{JWLXZ} and Wang \cite{YQW}. Since there are two singular terms in the twisted conical K\"ahler-Ricci flow \eqref{CK}, which are the initial metic $\hat\omega$ and current term $[D]$, we approximate flow \eqref{CK} by using the smooth twisted K\"ahler-Ricci flow
\begin{equation}\label{TK}
\left\{
\begin{aligned}
&\ \frac{\partial \omega_{\gamma\ve j}(t)}{\partial t}=-{\rm Ric}(\omega_{\gamma\ve j}(t))+(1-\gamma)\theta_{\ve}+\eta\\
 &\ \ \ \ \ \ \ \ \ \ \ \ \ \ \ \ \ \ \ \ \ \ \ \ \ \ \ \ \ \ \ \ \ \ \ \ \ \ \ \ \ \ \ \ \ \ \ \ \ \ \ \ \ \ \ \ \ \ \ , \\
 &\ \omega_{\gamma\ve j}(t)|_{t=0}=\omega_j
\end{aligned}
\right.\tag{$TKRF^{\eta}_{\gamma\ve j}$}
\end{equation}
where $\theta_\ve=\theta+\sqrt{-1}\partial\bar\partial\log(\ve^2+|s|_h^2)$ is smooth closed $(1,1)$-form approximating $[D]$ as $\ve$ goes to $0$, $s$ is the definition section of $D$, $h$ is a smooth Hermitian metric on $L_D$ with curvature $\theta$, and $\{\omega_j\}$ is a sequence of smooth K\"ahler metrics approximating $\hat\omega$ as $j$ goes to infinity.

The main steps are deriving uniform estimates that are independent of $j$ and $\ve$ on the smooth twisted K\"ahler-Ricci flows \eqref{TK}. By using Guedj-Zeriahi's arguments in \cite{VGAZ2}, we can show that the uniform estimates are independent of $j$. The key steps in this paper are deriving uniform estimates also independent of $\ve$. After deriving these uniform estimates, we prove Theorem \ref{0131002} by letting $j$ go to infinity and then $\ve$ go to zero. For the uniqueness, we extend Di Nezza-Lu's idea on smooth twisted K\"ahler-Ricci flow to the conical case by using the techniques introduced by the authors \cite{JWLXZ19} which ensure that the auxiliary functions achieve their extremums outside the divisor.

The paper is organized as follows. We first recall some results on the smooth twisted K\"ahler-Ricci flows, and give the definitions of existence and uniqueness of the twisted conical K\"ahler-Ricci flow in section \ref{Preliminaries}. Then we derive the uniform estimates for the smooth twisted K\"ahler-Ricci flows and then obtain the existence in section \ref{Existence}. At last section, we prove the uniqueness for twisted conical K\"ahler-Ricci flow.

\medskip

{\bf Acknowledgements.} J.W. Liu is grateful to Professor Miles Simon for his kind guidance on geometric flow and partial differential equations when he worked in his group at Otto-von-Guericke-University Magdeburg.

\section{Preliminaries}\label{Preliminaries}

In this section, we recall some results on the twisted K\"ahler-Ricci flows in \cite{VGAZ2, NL2017}, and then give the definitions of existence and uniqueness of the conical K\"ahler-Ricci flows. 

For the smooth twisted K\"ahler-Ricci flow, Guedj-Zeriahi gave the following regularizing property.

\begin{thm}[Theorem $\rm A$ in \cite{VGAZ2}]\label{0214001}
Let $\eta$ be a smooth closed $\left(1,1\right)$-form and $\tilde\omega\in[\omega]$ be a positive current with zero Lelong number. There exists a unique maximal family $\left(\omega_t\right)_{0<t<T^\eta_{1,\max}}$ of K\"ahler forms such that
\begin{equation}\label{0214002}
\frac{\partial\omega_t}{\partial t}=-{\rm Ric}\left(\omega_t\right)+\eta
\end{equation}
and $\omega_t$ weakly converges towards $\tilde\omega$ as $t\searrow0^+$.
\end{thm}
\begin{rem}\label{0214003}
We denote $\tilde\omega=\omega+\dd\tilde\varphi$ and write flow \eqref{0214002} as the following parabolic Monge-Amp\`ere equation
\begin{equation}\label{0219001}
\frac{\partial \varphi_t}{\partial t}=\log\frac{(\theta_t+\dd \varphi_t)^n}{\omega^n}+h,
\end{equation}
where $\theta_t=\omega+t\eta-t\rm Ric(\Omega)$, $\Omega$ is a smooth volume form and $h\in \mathcal{C}^\infty\left(X,\mathbb{R}\right)$ is a smooth function such that $\Omega=e^{-h}\omega^n$. The existence, weak convergence back to $\tilde\omega$ and uniqueness in Theorem \ref{0214001} are in the following sense. 

There exists a unique of smooth strictly $\theta_t$-psh functions $\varphi_t$ satisfying equation \eqref{0219001} in $(0,T^{\eta}_{1,\max})$, which converges to $\tilde\varphi$ in $L^1(X)$-sense as $t\searrow0^+$, such that $\omega_t=\theta_t+\dd\varphi_t$ satisfies flow \eqref{0214002}, and it converges towards $\tilde\omega$ as $t\searrow0^+$ in the sense of currents.

The uniqueness should be understood in the following weak sense, if $\psi_t$ is another solution to the equation \eqref{0219001} in above sense, then it lies below $\varphi_t$ which is obtained by Guedj-Zeriahi, that is, $\varphi_t$ is the maximal solution.
\end{rem}
Then Di Nezza-Lu improved above uniqueness result.
\begin{thm}[Uniqueness of Theorem $1.3$ in \cite{NL2017} ]\label{0219002}
If the Lelong number of $\tilde\omega$ is zero. Then any weak solution to twisted \krf  \eqref{0214002} starting from $\tilde\omega$ is maximal. In other words, the flow \eqref{0214002} is unique.
\end{thm}
\begin{rem}\label{0219003}
The weak solution and maximal in Theorem \ref{0219002} are in the sense of  Remark \ref{0214003}.
\end{rem}
Next, we rewrite the twisted conical K\"ahler-Ricci flow 
\begin{equation}\label{CK}
\left\{
\begin{aligned}
 &\ \frac{\partial \omega_{\gamma}(t)}{\partial t}=-{\rm Ric}(\omega_{\gamma}(t))+(1-\gamma)[D]+\eta\\
 &\ \ \ \ \ \ \ \ \ \ \ \ \ \ \ \ \ \ \ \ \ \ \ \ \ \ \ \ \ \ \ \ \ \ \ \ \ \ \ \ \ \ \ \ \ \ \ \ \ \ \ \ \ \ \ \ \ \ \  \\
 &\ \omega_{\gamma}(t)|_{t=0}=\hat{\omega}
\end{aligned}
\right.\tag{$CKRF^{\eta}_{\gamma}$}
\end{equation}
as a complex Monge-Amp\`ere equation which admits conical singularity.

Fix $\gamma\in(0,1]$. For any $T<T^\eta_{\gamma,\max}$, there exists a smooth volume form $\Omega_\gamma$ such that
\begin{equation}\label{0219004}
\omega_{\gamma t}=\omega+t\left(-{\rm Ric}(\Omega_\gamma)+(1-\gamma)\theta+\eta\right)>0
\end{equation}
for all $t\in[0,T]$.

In fact, since $[\omega]+T(-c_1(X)+(1-\gamma)c_1(L_D)+[\eta])>0$, there exists a smooth closed $(1,1)$-form
\begin{equation}\label{0219005}
\kappa_\gamma\in[\omega]+T(-c_1(X)+(1-\gamma)c_1(L_D)+[\eta])
\end{equation}
satisfying $\kappa_\gamma>0$. We denote
\begin{equation}\label{0219005}
\omega_{\gamma t}=\omega+\frac{t}{T}\left(\kappa_\gamma-\omega\right)=(1-\frac{t}{T})\omega+\frac{t}{T}\kappa_\gamma.
\end{equation}
It is obvious that 
\begin{equation}\label{02190051}
\omega_{\gamma t}\in[\omega]+t\left(-c_1(X)+(1-\gamma)c_1(L_D)+[\eta]\right)\ \ \ \text{and}\ \ \  \omega_{\gamma t}>0
\end{equation}
for all $t\in[0,T]$. 

On the other hand, since
\begin{equation}\label{0219006}
-\frac{1}{T}\left(\kappa_\gamma-\omega\right)+(1-\gamma)\theta+\eta\in c_1(X)
\end{equation} 
is a smooth closed $(1,1)$-form, by Calabi-Yau theorem, there exists a smooth volume form $\Omega_\gamma$ such that
\begin{equation}\label{}
{\rm Ric}(\Omega_\gamma)=-\frac{1}{T}\left(\kappa_\gamma-\omega\right)+(1-\gamma)\theta+\eta,
\end{equation}
which implies that for any $t\in[0,T]$,
\begin{equation}\label{0219007}
\omega_{\gamma t}=\omega+t\left(-{\rm Ric}(\Omega_\gamma)+(1-\gamma)\theta+\eta\right).
\end{equation}
It is obvious that there is a constant $C>1$ depending only on $\gamma$, $T$, $\theta$, $\eta$ and $\omega$ such that
\begin{equation}\label{02190071}
\frac{1}{C}\omega\leqslant\omega_{\gamma t}\leqslant C\omega
\end{equation}
for all $t\in[0,T]$. We denote 
\begin{equation}\label{0219004}
\Omega_\gamma=e^{-h_\gamma}\omega^n\ \ \ \ \ \text{and}\ \ \ \ \ \omega_{\gamma t}=\omega+t \nu_\gamma,
\end{equation}
where $\nu_\gamma=-{\rm Ric}(\Omega_\gamma)+(1-\gamma)\theta+\eta$.

Assume that 
\begin{equation}\label{0220001}
\hat\omega=\omega+\dd\varphi_0,
\end{equation}
then the twisted conical \krf can be written as the following Monge-Amp\`ere equation
\begin{equation}\label{CC}
\left\{
\begin{aligned}
 &\ \frac{\partial \varphi_{\gamma}(t)}{\partial t}=\log\frac{\left(\omega_{\gamma t}+\dd\varphi_{\gamma}(t)\right)^n}{\omega^n}+h_\gamma+(1-\gamma)\log|s|_h^2\\
 &\ \ \ \ \ \ \ \ \ \ \ \ \ \ \ \ \ \ \ \ \ \ \ \ \ \ \ \ \ \ \ \ \ \ \ \ \ \ \ \ \ \ \ \ \ \ \ \ \ \ \ \ \ \ \ \ \ \ \ \ \ \ \ \ \ \ \ \ \ \ \ \ \ \ \ \ \ \ \ \ \ \ \ .\\
 &\ \varphi_{\gamma}(t)|_{t=0}=\varphi_0
\end{aligned}
\right.
\end{equation}
\begin{defi}\label{0220003}
By saying that a closed positive $(1,1)$-current $\omega$ with locally bounded potentials is a conical K\"ahler metric with cone angle $2\pi \gamma $ ($0<\gamma\leqslant1$) along $D$, we mean that $\omega$ is a smooth K\"ahler metric on $M\setminus D$. And near each point $p\in D$, there exists a local holomorphic coordinate $(z_{1}, \cdots, z_{n})$ in a neighborhood $U$ of $p$ such that $D=\{z_{n}=0\}$ and $\omega$ is asymptotically equivalent to the model conical metric
\begin{equation}\label{0220004}
 \sum_{j=1}^{n-1} \sqrt{-1}dz_{j}\wedge d\overline{z}_{j}+ \frac{\sqrt{-1} dz_{n}\wedge d\overline{z}_{n}}{|z_{n}|^{2(1-\gamma)}} \ \ \ on\ \ U.\end{equation}
\end{defi}
For $\gamma\in(0,1)$,
\begin{equation}\label{0220004}
\omega_\gamma=\omega+k\sqrt{-1}\partial\bar{\partial}|s|_h^{2\gamma}
\end{equation}
is a conical K\"ahler metric with cone angle $2\pi \gamma$ along D, which was given by Donaldson \cite{SD2}. 
\begin{defi}\label{04.5}We call $\omega_\gamma(t)$ a weak solution to the conical K\"ahler-Ricci flow \eqref{CK} on spacetime $(0,T^\eta_{\gamma,\max})\times X$ if it satisfies the following conditions.
\begin{itemize}
  \item  For any $0<\delta<T<T^\eta_{\gamma,\max}$, there exists a constant $C$ such that for $t\in(\delta,T)$,
\begin{equation}
C^{-1}\omega_\gamma\leqslant\omega_\gamma(t)\leqslant C\omega_\gamma\ \ \ in\ \ \ X\setminus D;
\end{equation}
   \item  On $(0,T^\eta_{\gamma,\max})\times(X\setminus D)$, $\omega_\gamma(t)$ satisfies the smooth twisted K\"ahler-Ricci flow
  \begin{equation}\label{0221001}
\frac{\partial \omega_{\gamma}(t)}{\partial t}=-{\rm Ric}(\omega_{\gamma}(t))+\eta;
\end{equation}
  \item On $(0,T^\eta_{\gamma,\max})\times X$, $\omega_\gamma(t)$ satisfies equation \eqref{CK} in the sense of currents;
  \item There exists a metric potential $\varphi_\gamma(t)\in \mathcal{C}^{2,\alpha,\gamma}\left( X\right)\cap \mathcal{C}^{\infty}\left(X\setminus D\right)$ such that $\omega_\gamma(t)=\omega_{\gamma t}+\sqrt{-1}\partial\bar{\partial}\varphi_\gamma(t)$, $\lim\limits_{t\rightarrow0^{+}}\|\varphi_\gamma(t)-\varphi_{0}\|_{L^{1}(M)}=0$ and $\lim\limits_{t\rightarrow0^{+}}\varphi_\gamma(t)=\varphi_{0}$ in $X\setminus D$;
  \end{itemize}
\end{defi}
\begin{rem}\label{14}  In Definition \ref{04.5}, by saying that $\omega_\gamma(t)$ satisfies equation \eqref{CK} in the sense of currents on $X_{T}:=(0,T^\eta_{\gamma,\max})\times X$, we mean that for any smooth $(n-1,n-1)$-form $\zeta(t)$ with compact support in $X_T$, we have
  \begin{equation}\label{022100100}
  \int_{X_{T}}\frac{\partial \omega_\gamma(t)}{\partial t}\wedge \zeta(t,x)dt=\int_{X_{T}}(-Ric(\omega_\gamma(t))+(1-\gamma)[D]+\eta)\wedge \zeta(t,x)dt,
  \end{equation}
  where the integral on the left side can be written as
  \begin{equation}\label{0221001001}
  \int_{X_{T}}\frac{\partial \omega_\gamma(t)}{\partial t}\wedge \zeta(t,x)dt=-\int_{X_{T}} \omega_\gamma(t)\wedge\frac{\partial\zeta(t,x)}{\partial t}dt.
  \end{equation}
\end{rem}
\begin{defi}\label{0221001002}
We call $\phi(t)$ a weak solution to equation \eqref{CC} if $\omega_{\gamma t}+\dd\phi(t)$ is a weak solution to the conical K\"ahler-Ricci flow \eqref{CK} in the sense of Definition \ref{04.5}.
\end{defi}
\begin{rem}\label{0502001}
From Definitions \ref{04.5} and \ref{0221001002}, for any $0<\delta<T<T^\eta_{\gamma,\max}$, if $\phi(t)$ is a weak solution to equation \eqref{CC}, there exists a constant $C$ such that
\begin{equation}\label{031800399}
C^{-1}\omega_\gamma\leqslant\omega_{\gamma t}+\sqrt{-1}\partial\bar{\partial}\phi(t)\leqslant C\omega_\gamma.
\end{equation}
Hence
\begin{equation}\label{031800499}
\frac{\partial \phi(t)}{\partial t}=\log\frac{\left(\omega_{\gamma t}+\dd\phi(t)\right)^n}{\omega^n}+h_\gamma+(1-\gamma)\log|s|_h^2
\end{equation}
is uniformly bounded in $[\delta,T]\times (X\setminus D)$, which implies that
\begin{equation}\label{031800599}
\|\phi(t)-\phi(t')\|_{L^\infty(X)}\leqslant C|t-t'|\ \ \ \text{for all}\ t,t'\in[\delta,T].
\end{equation}
Therefore, for any $t>0$, there holds
\begin{equation}\label{0502002}
\lim\limits_{s\to t^+}\|\phi(s)-\phi(t)\|_{L^\infty(X)}=0.
\end{equation}
\end{rem}
There are two singular terms in the twisted conical \krf \eqref{CK}, which are the initial metic $\hat\omega$ and current term $[D]$. We choose a sequence of smooth K\"ahler metrics $\omega_j$ approximating to $\hat\omega$ as in \cite{VGAZ2}, that is, by Demailly's regularization result \cite{JPD}, there exists a sequence of smooth strictly $\omega$-psh functions $\varphi_{j}$ decreasing to $\varphi_0$ as $j\nearrow\infty$, and 
\begin{equation}\label{0220002}
\omega_j=\omega+\dd\varphi_j.
\end{equation}
For the current term $[D]$, as in \cite{JWLXZ, YQW}, we approximate it by $\theta_\ve=\theta+\dd\log\left(\ve^2+|s|_h^2\right)$ as $\ve\searrow0$. Therefore, the smooth twisted K\"ahler-Ricci flow approximating to the twisted conical \krf is
\begin{equation}\label{TK}
\left\{
\begin{aligned}
&\ \frac{\partial \omega_{\gamma\ve j}(t)}{\partial t}=-{\rm Ric}(\omega_{\gamma\ve j}(t))+(1-\gamma)\theta_{\ve}+\eta\\
 &\ \ \ \ \ \ \ \ \ \ \ \ \ \ \ \ \ \ \ \ \ \ \ \ \ \ \ \ \ \ \ \ \ \ \ \ \ \ \ \ \ \ \ \ \ \ \ \ \ \ \ \ \ \ \ \ \ \ \ ,\\
 &\ \omega_{\gamma\ve j}(t)|_{t=0}=\omega_j
\end{aligned}
\right.\tag{$TKRF^{\eta}_{\gamma\ve j}$}
\end{equation}
which is equivalent to the following complex Monge-Amp\`ere equation
\begin{equation}\label{TC}
\left\{
\begin{aligned}
 &\ \frac{\partial \varphi_{\gamma\ve j}(t)}{\partial t}=\log\frac{\left(\omega_{\gamma t}+\dd\varphi_{\gamma\ve j}(t)\right)^n}{\omega^n}+h_\gamma+(1-\gamma)\log\left(\ve^2+|s|_h^2\right)\\
 &\ \ \ \ \ \ \ \ \ \ \ \ \ \ \ \ \ \ \ \ \ \ \ \ \ \ \ \ \ \ \ \ \ \ \ \ \ \ \ \ \ \ \ \ \ \ \ \ \ \ \ \ \ \ \ \ \ \ \ \ \ \ \ \ \ \ \ \ \ \ \ \ \ \ \ \ \ \ \ \ \ \ \ \ \ \ \ \ \ \ \ \ \ .\\
 &\ \varphi_{\gamma\ve j}(t)|_{t=0}=\varphi_{j}
\end{aligned}
\right.
\end{equation}

In order to show that $\omega_\gamma(t)$ is a conical K\"ahler metric with cone angle $2\pi\gamma$ along $D$ at any positive time, we should prove that $\omega_\gamma(t)$ is equivalent to the conical K\"ahler metric $\omega_\gamma$. Therefore, in smooth approximation, we need to prove that $\omega_{\gamma\ve j}(t)$ is equivalent to the smooth K\"ahler metric $\omega_{\gamma\ve}$, which approximates $\omega_\gamma$ as $\ve\searrow0$. We choose $\omega_{\gamma\ve}$ to be Campana-Guenancia-P$\breve{a}$un's metric (equation $(18)$ in \cite{CGP}), that is,
\begin{equation}\label{0227001}
\omega_{\gamma\ve}=\omega+k\dd\chi_{\gamma\ve}=\omega+k\frac{\left\langle D's,D's\right\rangle}{(\ve^2+|s|_h^2)^{1-\gamma}}-\frac{k}{\gamma}\left((\varepsilon^2+|s|_h^2)^\gamma-\varepsilon^{2\gamma}\right)\theta,
\end{equation}
where $k\in\mathbb{R}$ is a positive constant, and 
\begin{equation}\label{0227002}
\chi_{\gamma\ve}=\frac{1}{\gamma}\int_0^{|s|_h^2}\frac{(\varepsilon^2+r)^\gamma-\varepsilon^{2\gamma}}{r}dr.
\end{equation}
From equation $(15)$ in \cite{CGP}, we know that there exists a constant independent of $\ve$ such that
\begin{equation}\label{02270021}
0\leqslant\chi_{\gamma\ve}\leqslant C.
\end{equation}
If $k$ is sufficiently small, then
\begin{equation}\label{0227003}
\omega_{\gamma\ve}\geqslant \frac{1}{2}\omega
\end{equation}
for all $\ve\in[0,1]$. Furthermore, we can choose a smaller $k$ such that
\begin{equation}\label{0227004}
\frac{1}{2}\omega-\frac{\left(C-\frac{1}{2}\right)k}{\gamma}\left((\varepsilon^2+|s|_h^2)^\gamma-\varepsilon^{2\gamma}\right)\theta\geqslant0
\end{equation}
for all $\ve\in[0,1]$, where $C$ is the constant in \eqref{02190071}. Then by using \eqref{02190071}, \eqref{0227001}, \eqref{0227003} and \eqref{0227004}, we have
\begin{equation}\label{0227005}
\omega_{\gamma t\ve}=\omega_{\gamma t}+k\dd\chi_{\gamma\ve}\leqslant(2C-1)\omega_{\gamma\ve},
\end{equation}
and
\begin{equation}\label{02270066}
\begin{aligned}
\omega_{\gamma t\ve}&\geqslant\frac{1}{C}\left(\omega+Ck\dd\chi_{\gamma\ve}\right)\\
&=\frac{1}{C}\left(\frac{1}{2}\omega+\frac{k}{2}\dd\chi_{\gamma\ve}+\frac{1}{2}\omega+\left(C-\frac{1}{2}\right)k\frac{\left\langle D's,D's\right\rangle}{(\ve^2+|s|_h^2)^{1-\gamma}}-\frac{\left(C-\frac{1}{2}\right)k}{\gamma}\left((\varepsilon^2+|s|_h^2)^\gamma-\varepsilon^{2\gamma}\right)\theta\right)\\
&\geqslant\frac{1}{2C}\omega_{\gamma\ve},
\end{aligned}
\end{equation}
that is, there is a constant $C$ depending only on $\gamma$, $T$, $\theta$, $\eta$ and $\omega$ such that
\begin{equation}\label{0227005}
\frac{1}{C}\omega_{\gamma\ve}\leqslant\omega_{\gamma t\ve}\leqslant C\omega_{\gamma\ve}.
\end{equation}
We choose $\omega_{\gamma t\ve}$ being the background metric, then equation \eqref{TC} is equivalent to equation
\begin{equation}\label{TC1}
\left\{
\begin{aligned}
 &\ \frac{\partial \phi_{\gamma\ve j}(t)}{\partial t}=\log\frac{\left(\omega_{\gamma t\ve}+\dd\phi_{\gamma\ve j}(t)\right)^n}{\omega_{\gamma\ve}^n}+F_{\gamma\ve}\\
 &\ \ \ \ \ \ \ \ \ \ \ \ \ \ \ \ \ \ \ \ \ \ \ \ \ \ \ \ \ \ \ \ \ \ \ \ \ \ \ \ \ \ \ \ \ \ \ \ \ \ \ \ \ \ \ \ \ \ \ \ \ \ \ \ \ \ \ \ ,\\
 &\ \phi_{\gamma\ve j}(t)|_{t=0}=\varphi_{j}-k\chi_{\gamma\ve}
\end{aligned}
\right.
\end{equation}
where $\phi_{\gamma\ve j}(t)=\varphi_{\gamma\ve j}(t)-k\chi_{\gamma\ve}$ and $F_{\gamma\ve}=\log\frac{\omega_{\gamma\ve}^n\left(\ve^2+|s|_h^2\right)^{1-\gamma}}{\omega^n}+h_\gamma$. In fact, there exists a uniform constant $C$ independent of $\ve$ such that
\begin{equation}\label{0227006}
|F_{\gamma,\ve}|\leqslant C\ \ \ \ \ \text{on}\ \ \ X
\end{equation}
for all $\ve\in(0,1)$ (equation $(25)$ in \cite{CGP}). In the following sections, in order to show the uniform estimates for the smooth twisted \krf \eqref{TK}, we need only to consider the equation \eqref{TC1}.

\section{Existence}\label{Existence}

In this section, we prove the uniform estimates for equation \eqref{TC1}, which does not depend on $j$ and $\ve$. Fix $\gamma\in(0,1)$, a smooth closed $(1,1)$-form $\eta$ and $0<T<T^\eta_{\gamma,\max}$. We first construct a super-solution to equation \eqref{TC1}, which gives the uniform upper bound to $\phi_{\gamma\ve j}(t)$. In the following, we say a constant being uniform if it does not depend on $j$ and $\ve$.

Since $\varphi_{j}$ is smooth strictly $\omega$-psh function and decreases to $\varphi_0$ as $j\nearrow\infty$, there exists a uniform constant $C$ such that
\begin{equation}\label{02280001}
\varphi_0\leqslant\varphi_j\leqslant C.
\end{equation}
Furthermore, combining with \eqref{02270021}, there holds
\begin{equation}\label{02280002}
\phi_{\gamma\ve j}(0)\leqslant C
\end{equation}
for some uniform constant $C$.

On the other hand, since $\varphi_0$ is only a $\omega$-psh function with zero Lelong number, there is no uniform lower bound for $\varphi_j$ and so for $\phi_{\gamma\ve j}(0)$.

\begin{lem}\label{0228001}
There exists a uniform constant $C$ depending only on $\gamma$, $T$, $\theta$, $\eta$ and $\omega$ such that
\begin{equation}\label{0228002}
\phi_{\gamma\ve j}(t)\leqslant\sup\limits_X\phi_{\gamma\ve j}(t_0)+C(t-t_0)
\end{equation}
for all $t_0\in[0,T]$, $j\geqslant1$, $\ve\in(0,1)$ and $t\in[t_0,T]$.
\end{lem}
\begin{rem}\label{0228123}
If we choose $t_0=0$, by \eqref{02280002} and Lemma \ref{0228001}, there exists a uniform constant $C$ depending only on $\gamma$, $T$, $\theta$, $\eta$ and $\omega$ such that
\begin{equation}\label{0228002}
\phi_{\gamma\ve j}(t)\leqslant C
\end{equation}
for all $j\geqslant1$, $\ve\in(0,1)$ and $t\in[0,T]$.
\end{rem}
\begin{proof}
Let $\psi_1(t)=\sup\limits_X\phi_{\gamma\ve j}(t_0)+C(t-t_0)$ for some uniform constant $C$. By \eqref{0227005} and \eqref{0227006}, there exist uniform constants $C$ such that
\begin{equation}\label{0228003}
\omega_{\gamma t\ve}+\dd\psi_1(t)=\omega_{\gamma t\ve}\leqslant C\omega_{\gamma\ve},
\end{equation}
and hence
\begin{equation}\label{0228004}
\log\frac{\left(\omega_{\gamma t\ve}+\dd\psi_1(t)\right)^n}{\omega_{\gamma\ve}^n}+F_{\gamma\ve}\leqslant C=\frac{\partial \psi_1(t)}{\partial t}.
\end{equation}
Combining that $\psi_1(t_0)\geqslant\phi_{\gamma\ve j}(t_0)$, by maximum principle, we have
\begin{equation}\label{0228005}
\phi_{\gamma\ve j}(t)\leqslant\psi_1(t)=\sup\limits_X\phi_{\gamma\ve j}(t_0)+C(t-t_0)
\end{equation}
for all $t_0\in[0,T]$, $j\geqslant1$, $\ve\in(0,1)$ and $t\in[t_0,T]$.
\end{proof}
\begin{lem}\label{0228006}
There exists a uniform constant $C$ depending only on $\gamma$, $T$, $\theta$, $\eta$ and $\omega$ such that
\begin{equation}\label{0228007}
\phi_{\gamma\ve j}(t)\geqslant\inf\limits_X\phi_{\gamma\ve j}(t_0)-C(t-t_0)
\end{equation}
for all $t_0\in[0,T]$, $j\geqslant1$, $\ve\in(0,1)$ and $t\in[t_0,T]$.
\end{lem}
\begin{rem}\label{0228008}
Since there is no uniform lower bound for $\phi_{\gamma\ve j}(0)$, we always consider the lower bound of $\phi_{\gamma\ve j}(t)$ on $[t_0,T]$ with positive $t_0$.
\end{rem}
\begin{proof}
Fix $t_0\in[0,T]$. Let $\psi_2(t)=\inf\limits_X\phi_{\gamma\ve j}(t_0)-C(t-t_0)$ for some constant $C$. Then by \eqref{0227005} and \eqref{0227006}, there exists a uniform constant $C$ independent of $j$, $\ve$ and $t_0$ such that
\begin{equation}\label{0228009}
\log\frac{(\omega_{\gamma t\ve}+\dd\psi_2(t))^n}{\omega_{\gamma\ve}^n}+F_{\gamma\ve}=\log\frac{\omega_{\gamma t\ve}^n}{\omega_{\gamma\ve}^n}+F_{\gamma\ve}\geqslant-C=\frac{\partial\psi_2(t)}{\partial t}.
\end{equation}
Since $\psi_2(t_0)=\inf\limits_X\phi_{\beta,\varepsilon, j}(t_0)\leq \phi_{\gamma\ve j}(t_0)$, by maximum principle, we have
\begin{equation}\label{0228010}
\phi_{\gamma\ve j}(t)\geqslant\psi_2(t)=\inf\limits_X\phi_{\gamma\ve j}(t_0)+C\left(t-t_0\right)
\end{equation}
for all $j\geqslant1$, $\ve\in(0,1)$, $t_0\in[0,T]$ and $t\in[t_0,T]$.
\end{proof}
Next, we show a upper bound for $\dot{\varphi}_{\gamma\ve j}(t):=\frac{\partial \varphi_{\gamma\ve j}(t)}{\partial t}$, which is important for deriving the uniform estimates for the oscillation of $\varphi_{\gamma\ve j}(t)$. The idea in this process follows from the one in \cite{VGAZ2}. Our contribution is to prove the estimates does not depend on $\ve$ and $j$.
\begin{lem}\label{0228011}
There exists a uniform constant $C$ depending only on $\sup\limits_X\varphi_0$, $\gamma$, $T$, $\theta$, $\eta$ and $\omega$ such that 
\begin{equation}\label{0228012}
\dot{\varphi}_{\gamma\ve j}(t)\leqslant-\frac{\varphi_0}{t}+\frac{C}{t}
\end{equation}
for all $j\geqslant1$, $\ve\in(0,1)$ and $t\in(0,T]$.
\end{lem}
\begin{proof}
Let $H(t,x)=t\dot{\varphi}_{\gamma\ve j}(t)-(\varphi_{\gamma\ve j}(t)-\varphi_{j})-nt$. Direct computations show that
\begin{equation}\label{0228013}
\frac{\partial}{\partial t} \dot{\varphi}_{\gamma\ve j}(t)=\tr_{\omega_{\gamma\ve j}(t)}\left(\nu_{\gamma}+\dd\dot{\varphi}_{\gamma\ve j}(t)\right),
\end{equation}
and hence
\begin{equation}\label{0228014}
\frac{\partial H}{\partial t}=t\ddot{\varphi}_{\gamma\ve j}(t)-n=t(\tr_{\omega_{\gamma\ve j}(t)}\nu_\gamma+\Delta_{\omega_{\gamma\ve j}(t)}\dot{\varphi}_{\gamma\ve j}(t))-n.
\end{equation}
On the other hand, we have
\begin{equation}\label{0228015}
\begin{aligned}
\Delta_{\omega_{\gamma\ve j}(t)} H&=t\Delta_{\omega_{\gamma\ve j}(t)} \dot{\varphi}_{\gamma\ve j}(t)-\Delta_{\omega_{\gamma\ve j}(t)} \varphi_{\gamma\ve j}(t)+\Delta_{\omega_{\gamma\ve j}(t)} \varphi_{j}\\
&=t\Delta_{\omega_{\gamma\ve j}(t)} \dot{\varphi}_{\gamma\ve j}(t)-n+\tr_{\omega_{\gamma\ve j}(t)}\omega_{\gamma t}+\Delta_{\omega_{\gamma\ve j}(t)} \varphi_{j}\\
&=t(\tr_{\omega_{\gamma\ve j}(t)}\nu_\gamma+\Delta_{\omega_{\gamma\ve j}(t)}\dot{\varphi}_{\gamma\ve j}(t))-n+\tr_{\omega_{\gamma\ve j}(t)}\omega_j.
\end{aligned}
\end{equation}
Then
\begin{equation}\label{0228016}
\dl H=-\tr_{\omega_{\gamma\ve j}(t)}\omega_j\leqslant0.
\end{equation}
By maximum principle, we have $H(t)\leqslant H(0)=0$. Therefore,
\begin{equation}\label{0228017}
\begin{aligned}
\dot{\varphi}_{\gamma\ve j}(t)&\leqslant \frac{\varphi_{\gamma\ve j}(t)-\varphi_{j}}{t}+n\\
&= \frac{\phi_{\gamma\ve j}(t)+k\chi_{\gamma\varepsilon}-\varphi_{j}}{t}+n\\
&\leqslant -\frac{\varphi_{0}}{t}+\frac{C}{t},
\end{aligned}
\end{equation}
for all $j\geqslant1$, $\ve\in(0,1)$ and $t\in(0,T]$, where we use \eqref{02270021}, \eqref{02280001} and Remark \ref{0228123}.
\end{proof}
\begin{rem}\label{022}
There exists a uniform constant $C$ depending only on $\gamma$, $T$, $\theta$, $\eta$ and $\omega$ such that
\begin{equation}
\dot{\phi}_{\gamma\ve j}(t)=\dot{\varphi}_{\gamma\ve j}(t)\leqslant \frac{\osc\limits_X\varphi_{\gamma\ve j}(t_0)+C}{t-t_0}\leqslant \frac{\osc\limits_X\phi_{\gamma\ve j}(t_0)+C}{t-t_0}
\end{equation}
for all $t_0\in[0,T]$, $j\geqslant1$, $\ve\in(0,1)$ and $t\in(t_0,T]$.

In fact, by similar arguments as in the proof of Lemma \ref{0228011}, we can also conclude that 
 \begin{equation}\label{023}
\dot{\varphi}_{\gamma\ve j}(t)\leqslant \frac{\varphi_{\gamma\ve j}(t)-\varphi_{\gamma\ve j}(t_0)}{t-t_0}+n
\end{equation}
for all $t_0\in[0,T]$, $j\geqslant1$, $\ve\in(0,1)$ and $t\in(t_0,T]$.

On the other hand, from Lemma \ref{0228001}, there exists a uniform constant $C$ depending only on $\gamma$, $T$, $\theta$, $\eta$ and $\omega$ such that
\begin{equation}\label{024}
\begin{aligned}
\varphi_{\gamma\ve j}(t)&=\phi_{\gamma\ve j}(t)+k\chi_{\gamma\ve}\\
&\leqslant\sup\limits_X\phi_{\gamma\ve j}(t_0)+C(t-t_0)+k\chi_{\gamma\ve}\\
&\leqslant\sup\limits_X\varphi_{\gamma\ve j}(t_0)+C
\end{aligned}
\end{equation}
for all $t_0\in[0,T]$, $j\geqslant1$, $\ve\in(0,1)$ and $t\in[t_0,T]$, where we use the fact that $\chi_{\gamma\ve}$ is non-negative in the last inequality. Therefore, by using \eqref{02270021}, we have
\begin{equation}\label{025}
\dot{\phi}_{\gamma\ve j}(t)=\dot{\varphi}_{\gamma\ve j}(t)\leqslant \frac{\osc\limits_X\varphi_{\gamma\ve j}(t_0)+C}{t-t_0}\leqslant \frac{\osc\limits_X\phi_{\gamma\ve j}(t_0)+C}{t-t_0}
\end{equation}
for all $t_0\in[0,T]$, $j\geqslant1$, $\ve\in(0,1)$ and $t\in(t_0,T]$.
 \end{rem}
Now we recall the integrability index. The integrability index of $\varphi_0$ at some point $x\in X$ is
\begin{equation}\label{0228018}
c(\varphi_0,x):=\sup\left\{c>0|e^{-2c\varphi_0}\in L^1(V_x)\right\},
\end{equation}
where $V_x$ denotes an arbitrary neighborhood of $x$. Denote
\begin{equation}\label{0228019}
c(\varphi_0)=\inf\left\{c(\varphi_0,x)| x\in X\right\}
\end{equation}
being the uniform integrability index of $\varphi_0$. It follows from Skoda's integrability theorem (see \cite{Sko72} for Skoda's theorem and \cite{Zer01} for a uniform version) that
\begin{equation}\label{0228020}
\frac{1}{\nu(\varphi_0,x)}\leq c(\varphi_0,x)\leq\frac{n}{\nu(\varphi_0,x)},
\end{equation}
where $\nu(\varphi_0,x)$ is the Lelong number of $\varphi_0$ at $x$. Thus $c(\varphi_0)=+\infty$ if and only if $\varphi_0$ has zero Lelong number at all points. 
	
Observe that $\varphi_{\gamma\ve j}(t)$ is a family of $\omega_{\gamma t}$-psh functions such that
\begin{equation}\label{0228021}
(\omega_{\gamma t}+\dd\varphi_{\gamma\ve j}(t))^n=e^{\dot{\varphi}_{\gamma\ve j}(t)-h_\gamma}\frac{\omega^n}{(\varepsilon^2+\vert s\vert_h^2)^{1-\gamma}}.
\end{equation}
By using Lemma \ref{0228011}, we have
\begin{equation}\label{0228022}
(\omega_{\gamma t}+\dd\varphi_{\gamma\ve j}(t))^n\leqslant e^{-\frac{\varphi_{0}}{t}+\frac{C}{t}}\frac{\omega^n}{(\varepsilon^2+\vert s\vert_h^2)^{1-\gamma}}.
\end{equation}
Since $\varphi_0$ has zero lelong number at all points, for each fixed $t>0$, the right hand side of \eqref{0228021} is uniformly bounded in $L^p$-sense for some $p>1$. Then the oscillation of $\varphi_{\gamma\ve j}(t)$ is uniformly bounded, which follows from Ko{\l}odziej's $L^p$-estimates \cite{K000} (see also \cite{EGZ}).
\begin{thm}\label{0228023}
For each $t>0$, there exists a uniform constant $M(t)>0$ depending only on $\sup\limits_X\varphi_0$, $\gamma$, $T$, $\theta$, $\eta$ and $\omega$ such that  such that for all $j\geqslant1$ and $\ve\in(0,1)$,
\begin{equation}\label{0228024}
\osc\limits_X\varphi_{\gamma\ve j}(t)\leqslant M(t).
\end{equation}
\end{thm}
\begin{rem}\label{0228025}
Combining with \eqref{02270021}, it is obvious that for each $t>0$, there exists a uniform constant $M(t)>0$ depending only on $\sup\limits_X\varphi_0$, $\gamma$, $T$, $\theta$, $\eta$ and $\omega$ such that 
\begin{equation}\label{0228026}
\osc\limits_X\phi_{\gamma\ve j}(t)\leqslant M(t)
\end{equation}
for all $j\geqslant1$ and $\ve\in(0,1)$.
\end{rem}
Next, we give a lower bound for $\dot{\phi}_{\gamma\ve j}(t)$ by using the arguments in \cite{JWLXZ1, JSGT}.
\begin{pro}\label{0228027}
For any $0<t_0<T'<T$, there exists a uniform constant $C$ depending only on $\sup\limits_X\varphi_0$, $\gamma$, $T'$, $T$, $\theta$, $\eta$ and $\omega$ such that
\begin{equation}\label{0228018}
\dot{\phi}_{\gamma\ve j}(t)\geqslant n\log(t-t_0)-A\osc\limits_X\phi_{\gamma\ve j}(t_0)-C
\end{equation}
for all $j\geqslant1$, $\ve\in(0,1)$ and $t\in(t_0,T']$, where $A$ is a uniform constant satisfying $A>\frac{1}{T-T'}$.
\end{pro}

\begin{proof}
Let $G=\dot{\phi}_{\gamma\ve j}(t)+A\phi_{\gamma\ve j}-n\log(t-t_0)$. Direct computations show that
\begin{equation}\label{0229001}
\begin{aligned}
\dl G&=\ddot{\phi}_{\gamma\ve j}(t)+A\dot{\phi}_{\gamma\ve j}(t)-\frac{n}{t-t_0}-\Delta_{\omega_{\gamma\ve j}(t)}\dot{\phi}_{\gamma\ve j}(t)-A\Delta_{\omega_{\gamma\ve j}(t)}\phi_{\gamma\ve j}(t)\\
&=\tr_{\omega_{\gamma\ve j}(t)}\left(\nu_\gamma+\dd\dot{\phi}_{\gamma\ve j}(t)\right)+A\dot{\phi}_{\gamma\ve j}(t)-\frac{n}{t-t_0}\\
&\ \ \ \ -\Delta_{\omega_{\gamma\ve j}(t)}\dot{\phi}_{\gamma\ve j}(t)-An+A\tr_{\omega_{\gamma\ve j}(t)}\omega_{\gamma t\varepsilon}\\
&=A\dot{\phi}_{\gamma\ve j}(t)+\tr_{\omega_{\gamma\ve j}(t)}(A\omega_{\gamma t\varepsilon}+\nu_\gamma)-\frac{n}{t-t_0}-An.
\end{aligned}
\end{equation}
Choosing a uniform constant $A>\frac{1}{T-T'}$, for $t\in(t_0,T']$, we have $t+\frac{1}{A}<t+T-T'<T$. Hence by \eqref{02270066},
\begin{equation}\label{0229002}
\begin{aligned}
A\omega_{\gamma t\ve}+\nu_\gamma&=A(\omega+(t+\frac{1}{A})\nu_\gamma+k\dd\chi_{\gamma\ve})\\
&=A(\omega_{\gamma (t+\frac{1}{A})}+ k\dd\chi_{\gamma\varepsilon})\\
&\geqslant\frac{A}{C}\omega_{\gamma\varepsilon},
\end{aligned}
\end{equation}
and by using mean value inequality, we have
\begin{equation}\label{0229003}
\tr_{\omega_{\gamma\ve j}(t)}(A\omega_{\gamma t\ve}+\nu_\gamma)\geqslant\frac{A}{C}\tr_{\omega_{\gamma\ve j}(t)}\omega_{\gamma\ve}\geqslant\frac{An}{C}\left(\frac{\omega_{\gamma\ve}^n}{\omega_{\gamma\ve j}^n(t)}\right)^{\frac{1}{n}}.
\end{equation}
Therefore, we conclude that
\begin{equation}\label{0229004}
\begin{aligned}
\dl G&\geqslant A\dot{\phi}_{\gamma\ve j}(t)+\frac{An}{C}\left(\frac{\omega_{\gamma\ve}^n}{\omega_{\gamma\ve j}^n(t)}\right)^{\frac{1}{n}}-\frac{n}{t-t_0}-An\\
&=\frac{An}{C}\left(\frac{\omega_{\gamma\ve}^n}{\omega_{\gamma\ve j}^n(t)}\right)^\frac{1}{n}+A\log\frac{\omega_{\gamma\ve j}^n(t)}{\omega_{\gamma\ve}^n}+AF_{\gamma\ve}-\frac{n}{t-t_0}-An\\
&\geqslant\frac{An}{C}\left(\frac{\omega_{\gamma\ve}^n}{\omega_{\gamma\ve j}^n(t)}\right)^\frac{1}{n}+A\log\frac{\omega_{\gamma\ve j}^n(t)}{\omega_{\gamma\ve}^n}-\frac{C}{t-t_0}.
\end{aligned}
\end{equation}
Let $(t_1,x_1)$ be the minimum point of $G$ on $[t_0, T']\times X$. Since $G$ tends to $+\infty$ as $t\to0^+$, $t_1>t_0$. At $(t_1,x_1)$, if $\frac{\omega_{\gamma\ve}^n}{\omega_{\gamma\ve j}^n(t)}\leqslant1$, we have 
\begin{equation}\label{0229005}
\dot{\phi}_{\gamma\ve j}(t_1,x_1)\geqslant F_{\gamma\varepsilon}(x_1)\geqslant-C,
\end{equation}
where $C$ is the uniform constant in \eqref{0227006}.

Since function $\frac{An}{2C}x-An\log x$ goes to $+\infty$ as $x$ goes to $+\infty$, without loss of generality, we assume that at $(t_1,x_1)$,
\begin{equation}\label{0229006}
\frac{\omega_{\gamma\ve}^n}{\omega_{\gamma\ve j}^n(t)}>1\ \ \ \text{and}\ \ \ \frac{An}{2C}\left(\frac{\omega_{\gamma\ve}^n}{\omega_{\gamma\ve j}^n(t)}\right)^\frac{1}{n}+A\log\frac{\omega_{\gamma\ve j}^n(t)}{\omega_{\gamma\ve}^n}\geqslant0.
\end{equation}
Then at $(t_1,x_1)$, we have
\begin{equation}\label{0229007}
0\geqslant\dl G\geqslant\frac{An}{2C}\left(\frac{\omega_{\gamma\ve}^n}{\omega_{\gamma\ve j}^n(t)}\right)^\frac{1}{n}-\frac{C}{t_1-t_0},
\end{equation}
which implies that 
\begin{equation}\label{0229008}
\dot{\phi}_{\gamma\ve j}(t_1,x_1)\geqslant n\log(t_1-t_0)-C.
\end{equation}
So on $[t_0, T']\times X$, we have
\begin{equation}\label{0229009}
\dot{\phi}_{\gamma\ve j}(t)+A\phi_{\gamma\ve j}(t)-n\log(t-t_0)\geqslant A\phi_{\gamma\ve j}(x_1,t_1)-C.
\end{equation}
Combining  Lemma \ref{0228001} with Lemma \ref{0228006}, we have
\begin{equation}\label{0229010}
\begin{aligned}
 \dot{\phi}_{\gamma\ve j}&\geqslant A\inf\limits_X\phi_{\gamma\ve j}(t_0)-C(t-t_0)-A\sup_X\phi_{\gamma\ve j}(t_0)+n\log(t-t_0)\\
 &\geqslant n\log(t-t_0)-A\osc_X\phi_{\gamma\ve j}(t_0)-C.
\end{aligned}
\end{equation}
We complete the proof.
\end{proof}
Now we prove the uniform Laplacian $C^2$-estimate. 
\begin{pro}\label{0229011}
For any $0<t_0<T'<T$, there exists a uniform constant $C_1$ depending only on $\sup\limits_X\varphi_0$, $\gamma$, $T$, $\theta$, $\eta$ and $\omega$, and a uniform constant $C_2$ depending only on $t_0$, $T'$, $\sup\limits_X\varphi_0$, $\gamma$, $T$, $\theta$, $\eta$ and $\omega$, such that 
\begin{equation}\label{0229012}
(t-t_0)\log\tr_{\omega_{\gamma\ve}}\omega_{\gamma\ve j}(t)\leqslant C_1\osc\limits_X\phi_{\gamma\ve j}(t_0)+C_2
\end{equation}
for all $j\geqslant1$, $\ve\in(0,1)$ and $t\in(t_0,T']$.
\end{pro}
\begin{proof}
Set $K(t,x)=(t-t_0)\log\tr_{\omega_{\gamma\ve}}\omega_{\gamma\ve j}(t)+(t-t_0)\chi_{\gamma\varepsilon}-B\phi_{\gamma\ve j}(t)$, where $B$ is a uniform constant to be determined later. By using the computations in \cite{JWLXZ, JWLXZ1}, we have
\begin{equation}\label{0229013}
\begin{aligned}
\dl K&\leqslant\log\tr_{\omega_{\gamma\ve}}\omega_{\gamma\ve j}(t)+C\tr_{\omega_{\gamma\ve j}(t)}\omega_{\gamma\ve}-B\dot{\phi}_{\gamma\ve j}(t)-B\tr_{\omega_{\gamma\ve j}(t)}\omega_{\gamma t\ve}+C\\
&\leqslant(n-1)\log\tr_{\omega_{\gamma\ve j}(t)}\omega_{\gamma\ve}+C\tr_{\omega_{\gamma\ve j}(t)}\omega_{\gamma\ve}-(B-1)\log\frac{\omega_{\gamma\ve j}^n(t)}{\omega_{\gamma\ve}^n}-\frac{B}{C}\tr_{\omega_{\gamma\ve j}(t)}\omega_{\gamma\ve}+C\\
&\leqslant-\frac{B}{2C}\tr_{\omega_{\gamma\ve j}(t)}\omega_{\gamma\ve}+(n-1)\log\tr_{\omega_{\gamma\ve j}(t)}\omega_{\gamma\ve}-(B-1)\log\frac{\omega_{\gamma\ve j}^n(t)}{\omega_{\gamma\ve}^n}+C,
\end{aligned}
\end{equation}
where we choose $B$ large enough in last inequality, and we use inequality 
\begin{equation}\label{02290191}
\tr_{\omega_{\gamma\ve}}\omega_{\gamma\ve j}(t)\leqslant n\left(\frac{\omega^n_{\gamma\ve j}(t)}{\omega_{\gamma\ve}^n}\right)\left(\tr_{\omega_{\gamma\ve j}(t)}\omega_{\gamma\ve}\right)^{n-1},
\end{equation}
in the second inequality. 

Let $(t_1,x_1)$ be the maximum point of $K$ on $[t_0, T']\times X$. If $t_1=t_0$, it is obvious that there holds \eqref{0229012}. If $t_1>t_0$, since function $-\frac{B}{4C}x+(n-1)\log x$ goes to $-\infty$ as $x$ goes to $+\infty$, without loss of generality, we assume that at $(t_1,x_1)$, 
\begin{equation}\label{0229014}
-\frac{B}{4C}\tr_{\omega_{\gamma\ve j}(t)}\omega_{\gamma\ve}+(n-1)\log\tr_{\omega_{\gamma\ve j}(t)}\omega_{\gamma\ve}\leqslant0.
\end{equation}
Then at $(t_1,x_1)$, we have
\begin{equation}\label{0229015}
\dl K\leq-\frac{B}{4C}\tr_{\omega_{\gamma\ve j}(t)}\omega_{\gamma\ve}-(B-1)\dot{\phi}_{\gamma\ve j}(t)+C.
\end{equation}
By maximum principle, at $(t_1,x_1)$, it holds that
\begin{equation}\label{0229016}
\tr_{\omega_{\gamma\ve j}(t)}\omega_{\gamma\ve}\leqslant -C\dot{\phi}_{\gamma\ve j}(t_1)+C.
\end{equation}
By using Proposition \ref{0228027}, at $(t_1,x_1)$, we have
\begin{equation}\label{0229017}
\tr_{\omega_{\gamma\ve j}(t)}\omega_{\gamma\ve}\leqslant-C\log(t-t_0)+C\osc_X\phi_{\gamma\ve j}(t_0)+C.
\end{equation}
Combining \eqref{02290191} with Remark \ref{022}, there holds
\begin{equation}\label{0229018}
\begin{aligned}
\tr_{\omega_{\gamma\ve}}\omega_{\gamma\ve j}(t)&\leqslant Ce^{\dot{\phi}_{\gamma\ve j}(t)}\left(\osc\limits_X\phi_{\gamma\ve j}(t_0)+\log\frac{1}{t-t_0}+1\right)^{n-1}\\
&\leqslant Ce^{\frac{\osc\limits_X\phi_{\gamma\ve j}(t_0)+C}{t-t_0}}\left(\osc\limits_X\phi_{\gamma\ve j}(t_0)+\log\frac{1}{t-t_0}+1\right)^{n-1}.
\end{aligned}
\end{equation}
Therefore, we have
\begin{equation}\label{0229019}
\begin{aligned}
&\ \ \ \ (t-t_0)\log\tr_{\omega_{\gamma\ve}}\omega_{\gamma\ve j}(t)+(t-t_0)\chi_{\beta,\varepsilon}-B\phi_{\gamma\ve j}(t)\\
&\leqslant(n-1)(t_1-t_0)\log\left(\osc\limits_X\phi_{\gamma\ve j}(t_0)+\log\frac{1}{t_1-t_0}+1\right)+\osc\limits_X\phi_{\gamma\ve j}(t_0)-B\phi_{\gamma\ve j}(t_1)+C.
\end{aligned}
\end{equation}
By using Lemma \ref{0228001} and Lemma \ref{0228006}, we get
\begin{equation}\label{0229020}
(t-t_0)\log\tr_{\omega_{\gamma\ve}}\omega_{\gamma\ve j}(t)\leqslant C_1\osc\limits_X\phi_{\gamma\ve j}(t_0)+C_2.
\end{equation}
We complete the proof.
\end{proof}

For any $0<t_0<T<T^\eta_{\gamma,\max}$, combining Remark \ref{0228025}, Proposition \ref{0228027} and Proposition \ref{0229011}, we have
\begin{equation}\label{0229021}
\left(t-\frac{t_0}{2}\right)\log\tr_{\omega_{\gamma\ve}}\omega_{\gamma\ve j}(t)\leqslant C_1\osc\limits_X\phi_{\gamma\ve j}(\frac{t_0}{2})+C_2\ \ \ \text{on}\ \ \ \left[\frac{t_0}{2}, T\right]
\end{equation}
and 
\begin{equation}\label{0229024}
\dot{\phi}_{\gamma\ve j}(t)\geqslant n\log\left(t-\frac{t_0}{2}\right)-A\osc\phi_{\gamma\ve j}\left(\frac{t_0}{2}\right)-C\ \ \ \text{on}\ \ \ \left[\frac{t_0}{2}, T\right]
\end{equation}
for some uniform constants $C_1$, $C_2$, $A$ and $C$. Therefore, 
\begin{equation}\label{0229022}
\tr_{\omega_{\gamma\ve}}\omega_{\gamma\ve j}(t)\leqslant C\ \ \ \text{and}\ \ \ \dot{\phi}_{\gamma\ve j}(t)\geqslant-C\ \ \ \text{on}\ \ \ \left[t_0, T\right],
\end{equation}
where the uniform constant $C$ depends only on $t_0$, $\sup\limits_X\varphi_0$, $\gamma$, $T$, $\theta$, $\eta$ and $\omega$. Then we have
\begin{equation}\label{0229023}
\tr_{\omega_{\gamma\ve j}(t)}\omega_{\gamma\ve}\leqslant n\left(\frac{\omega_{\gamma\ve}^n}{\omega^n_{\gamma\ve j}(t)}\right)\left(\tr_{\omega_{\gamma\ve}}\omega_{\gamma\ve j}(t)\right)^{n-1}\leqslant C\ \ \ \text{on}\ \ \ \left[t_0, T\right].
\end{equation}
Therefore, there exists a uniform constant $C$ such that 
\begin{equation}\label{0229025}
\frac{1}{C}\omega_{\gamma\ve}\leqslant\omega_{\gamma\ve j}(t)\leqslant C\omega_{\gamma\ve}\ \ \ \text{on}\ \ \ \left[t_0, T\right]\times X.
\end{equation}
Furthermore, on any $\left[t_0, T\right]\times K$ for $K\subset\subset X\setminus D$, there exists a uniform constant $C$ such that
\begin{equation}\label{0229026}
\frac{1}{C}\omega\leqslant\omega_{\gamma\ve j}(t)\leqslant C\omega.
\end{equation}
Then Evans-Krylov-Safonov's estimates (see \cite{NVK}) imply the high order estimates
\begin{pro}\label{2171} For any $0<t_0<T<T^\eta_{\gamma,\max}$ and $k\in\mathbb{N}^{+}$, there exists a uniform constant $C$ depending only on $\ve$, $t_0$, $\sup\limits_X\varphi_0$, $\gamma$, $T$, $\theta$, $\eta$ and $\omega$, such that for any $j\geqslant1$,
\begin{equation}
\|\varphi_{\gamma\ve j}(t)\|_{C^{k}\big([t_0,T]\times X\big)}\leqslant C.
\end{equation}
\end{pro}
\begin{rem}\label{0229027}
Fix $\ve\in(0,1)$, by letting $j\to\infty$, we know that there is a smooth strictly $\omega$-psh $\varphi_{\gamma\ve}(t)$ such that $\omega_{\gamma\ve}(t)=\omega+\dd\varphi_{\gamma\ve}(t)$ is a smooth solution to equation
\begin{equation}\label{0229028}
\left\{
\begin{aligned}
&\ \frac{\partial \omega_{\gamma\ve }(t)}{\partial t}=-{\rm Ric}(\omega_{\gamma\ve }(t))+(1-\gamma)\theta_{\ve}+\eta\\
 &\ \ \ \ \ \ \ \ \ \ \ \ \ \ \ \ \ \ \ \ \ \ \ \ \ \ \ \ \ \ \ \ \ \ \ \ \ \ \ \ \ \ \ \ \ \ \ \ \ \ \ \ \ \ \ \ \ \ \ \\
 &\ \omega_{\gamma\ve }(t)|_{t=0}=\hat\omega
\end{aligned}
\right.\tag{$TKRF^{\eta}_{\gamma\ve j}$}
\end{equation}
on $(0,T^\eta_{\gamma,\max})\times X$. By using Guedj-Zeriahi's arguments in \cite{VGAZ2} and Theorem \ref{0219002}, we also conclude that $\omega_{\gamma\ve }(t)$ is unique, $\varphi_{\gamma\ve}(t)$ converges to $\varphi_0$ in $L^1$-sense as $t\searrow0^+$, and so $\omega_{\gamma\ve }(t)$ converges $\tilde\omega$ as $t\searrow0^+$ in the sense of currents.

Furthermore, similar arguments as in Proposition $3.3$ of \cite{JWLXZ1} imply that for any positive $t$, $\varphi_{\gamma\ve}(t)$ is decreasing as $\ve\searrow0^+$.
\end{rem}
\begin{pro}\label{217} For any $0<t_0<T<T^\eta_{\gamma,\max}$, $k\in\mathbb{N}^{+}$ and $K\subset\subset X\setminus D$, there exists a uniform constant $C$ depending only on $t_0$, $\sup\limits_X\varphi_0$, $\gamma$, $T$, $\theta$, $\eta$, $\omega$ and $dist_{\omega}(K,D)$, such that for any $\varepsilon\in(0,1)$ and $j\geqslant1$,
\begin{equation}
\|\varphi_{\gamma\ve j}(t)\|_{C^{k}\big([t_0,T]\times K\big)}\leqslant C.
\end{equation}
\end{pro}
For any $[t_0, T]\times K\subset\subset (0,T^\eta_{\gamma,\max})\times (X\setminus D)$ and $k\in\mathbb{N}^{+}$, $\|\varphi_{\gamma\varepsilon j}(t)\|_{C^{k}([t_0,T]\times K)}$ is uniformly bounded by Proposition \ref{217}. Let $j$ approximate to $\infty$, we obtain that $\|\varphi_{\gamma\varepsilon}(t)\|_{C^{k}([\delta,T]\times K)}$ is also uniformly bounded. Then let $t_0$ approximate to $0$, $T$ approximate to $\infty$ and $K$ approximate to $X\setminus D$, by diagonal rule, we get a sequence $\{\varepsilon_i\}$, such that $\varphi_{\gamma\varepsilon_i}(t)$ converges in $C^\infty_{loc}$-topology on $ (0,T^\eta_{\gamma,\max})\times (X\setminus D)$ to a function $\varphi_\gamma(t)\in C^\infty\left( (0,T^\eta_{\gamma,\max})\times (X\setminus D)\right)$ and satisfies equation
\begin{equation}\label{223}  
\frac{\partial \varphi_{\gamma}(t)}{\partial t}=\log\frac{(\omega_{\gamma t}+\sqrt{-1}\partial\bar{\partial}\varphi_\gamma(t))^{n}}{\omega^{n}}+h_{\gamma}+(1-\gamma)\log|s|_{h}^{2}
\end{equation}
in $ (0,T^\eta_{\gamma,\max})\times (X\setminus D)$. Since $\varphi_{\gamma\varepsilon}(t)$ is monotone decreasing as $\varepsilon\searrow0^+$, we conclude that $\varphi_{\gamma\varepsilon}(t)$ converges in $C^\infty_{loc}$-topology of $ (0,T^\eta_{\gamma,\max})\times (X\setminus D)$ to $\varphi_\gamma(t)$. Combining all above arguments with \eqref{0229025}, we have that for any $0<\delta<T<T^\eta_{\gamma,\max}$, there exists a uniform constant $C$ such that for any $t\in(\delta,T)$, 
\begin{equation}\label{221}
\frac{1}{C}\omega_\gamma\leqslant\omega_\gamma(t)=\omega_{\gamma t}+\sqrt{-1}\partial\overline{\partial}\varphi_\gamma(t)\leqslant C\omega_\gamma\ \ \text{in}\ X\setminus D,
\end{equation}
which implies that $\omega_\gamma(t)$ is a conical K\"ahler metric with cone angle $2\pi\gamma$ along $D$. Moreover, by using similar arguments as in Theorem $3.6$ of \cite{JWLXZ}, we can also conclude that that $\omega_\gamma(t)$ satisfies equation \eqref{CK} in the sense of currents. 

For any positive $t_1$, we assume that $t_1\in[\delta,T]\subseteq(0,T^\eta_{\gamma,\max})$. If we choose $t_0=\frac{\delta}{2}$ in  Remark \ref{022} and Proposition \ref{0228027} and $t=\frac{\delta}{2}$ in Theorem \ref{0228023}, then
\begin{equation}\label{0301001}
\|\dot{\varphi}_{\gamma\ve}(t)\|_{C^0(X)}\leqslant C\ \ \ \text{on}\ \ \ [\delta,T]\times X
\end{equation}
for some uniform constant $C$. Hence
\begin{equation}\label{0301002}
|\varphi_{\gamma\ve}(t)-\varphi_{\gamma\ve}(t_1)|\leqslant C|t-t_1|\ \ \ \text{on}\ \ \ [\delta,T]\times X.
\end{equation}
Let $\ve\to0$, there holds
\begin{equation}\label{0301002}
\|\varphi_{\gamma}(t)-\varphi_{\gamma}(t_1)\|_{C^0(X)}\leqslant C|t-t_1|\ \ \ \text{on}\ \ \ [\delta,T].
\end{equation}
Therefore, $\omega_{\gamma}(t)$ is a solution to the twisted conical \krf \eqref{CK} with initial metric $\omega_{\gamma}(t_1)$ on $[t_1, T^\eta_{\gamma,\max})$ in the sense of Definition $1.1$ in \cite{JWLXZ1}, that is, the potential $\varphi_\gamma(t)$ converges to $\varphi_\gamma(t_1)$ in $L^\infty$-sense. Then by the authors' uniqueness and regularity results (Theorem $3.7$ and Theorem $3.10$ in \cite{JWLXZ1}), $\varphi_\gamma(t)\in\mathcal{C}^{2,\alpha,\gamma}(X)$. So next we only need to prove that $\varphi_\gamma(t)$ converges to $\varphi_{0}$ in $L^{1}$-sense on $X$.
\begin{pro}\label{0229029}
There holds
\begin{equation}\label{0229030}
\lim\limits_{t\rightarrow0^{+}}\|\varphi_\gamma(t)-\varphi_{0}\|_{L^{1}(M)}=0\ \ \ \text{and}\ \ \ \lim\limits_{t\rightarrow0^{+}}\varphi_\gamma(t)=\varphi_{0}\ \ \text{in}\ \ X\setminus D.
\end{equation}
\end{pro}
\begin{proof}
For the upper bound, since $\varphi_{\gamma\varepsilon}(t)$ is monotone decreasing as $\varepsilon\searrow0^+$ and $\varphi_{\gamma\varepsilon j}(t)$ is monotone decreasing as $j\to\infty$,
\begin{equation}\label{0229031}
\varphi_\gamma(t)-\varphi_0\leqslant\varphi_{\gamma\ve}(t)-\varphi_0\leqslant\varphi_{\gamma\ve j}(t)-\varphi_j+\varphi_j-\varphi_0
\end{equation}
and
\begin{equation}\label{0502003}
\varphi_\gamma(t)-\varphi_0\leqslant\varphi_{\gamma\ve j}(t)-\varphi_j+\varphi_j-\varphi_0.
\end{equation}

For the lower bound, fix $l>0$ such that $(2l-1)\omega+\nu_\gamma\geqslant0$. We consider equation
\begin{equation}\label{0229032}
(\omega+\dd\varphi_{\gamma\ve j})^n=e^{\varphi_{\gamma\ve j}-h_\gamma-2l\varphi_j}\frac{\omega^n}{\left(\ve^2+\vert s\vert^2_h\right)^{(1-\gamma)}}.
\end{equation}
By Ko{\l}odziej's $L^p$-estimates, there exists a uniform constant $C$ independent of $\ve$ and $j$ such that 
\begin{equation}\label{0229037}
\|\varphi_{\gamma\ve j}\|_{C^0(X)}\leqslant C.
\end{equation}
We set $\psi_t=(1-2lt)\varphi_j+t\varphi_{\gamma\ve j}+n(t\log t-t)$. For $0<t<\min\{\frac{1}{2l}, T^\eta_{\gamma,\max} \}$,  we have 
\begin{equation}\label{0229033}
\begin{split}
\omega_{\gamma t}+\dd\psi_t&=\omega+t\nu_\gamma+(1-2lt)\dd\varphi_j+t\dd\varphi_{\gamma\ve j}\\
&=(1-2lt)\omega_{\varphi_j}+t\omega_{\varphi_{\gamma\ve j}}+t((2l-1)\omega+\nu_\gamma)\\
&\geqslant t\omega_{\varphi_{\gamma\ve j}}.
\end{split}
\end{equation}
Hence
\begin{equation}\label{0229034}
(\omega_{\gamma t}+\dd\psi_t)^n\geqslant t^n\omega_{\varphi_{\gamma\ve j}}^n=t^ne^{\varphi_{\gamma\ve j}-h_\gamma-2l\varphi_j}\frac{\omega^n}{\left(\ve^2+\vert s\vert^2_h\right)^{(1-\gamma)}},
\end{equation}
which implies that
\begin{equation}\label{0229035}
\begin{split}
&\ \ \ \ \log\frac{(\omega_{\gamma t}+\dd\psi_t)^n}{\omega^n}+h_\gamma+(1-\gamma)\log(\varepsilon^2+\vert s\vert_h^2)\\
&\geqslant n\log t+\varphi_{\gamma\ve j}-2l\varphi_j=\frac{\partial\psi_t}{\partial t}.
\end{split}
\end{equation}
So $\psi_t$ is a subsolution to equation \eqref{TC} with initial value $\varphi_j$. Since $\varphi_j\geqslant\varphi_0$, by maximum principle, 
\begin{equation}\label{0229036}
\psi_t\leqslant \varphi_{\gamma\ve j}(t).
\end{equation}
We let $j\nearrow\infty$ and then $\ve\searrow0$, there holds
\begin{equation}\label{02290366}
(1-2lt)\varphi_0-Ct+n(t\log t-t)\leqslant\psi_t\leqslant \varphi_{\gamma}(t).
\end{equation}
Since $\varphi_{\gamma\ve}(t)$ converges to $\varphi_0$ in $L^1$-sense on $X$, $\varphi_{\gamma\ve j}(t)$ converges to $\varphi_j$ in $\mathcal{C}^\infty$-sense on $X$ and $\varphi_{j}$ converges to $\varphi_0$ pointwise in $X\setminus D$ as $t\to0$, combining \eqref{0229031}, \eqref{0502003} with \eqref{02290366}, we complete the proof.
\end{proof}

\section{Uniqueness}\label{Uniqueness}
In this section, we consider the uniqueness of the conical K\"ahler-Ricci flow \eqref{CK}. In order to overcome the problems from the singular term in the equation, we combine Di Nezza-Lu's idea in \cite{NL2017} with the techniques used by the authors in \cite{JWLXZ19}, which ensure the maximum point can be always attained outside the divisor.

In \cite{JWLXZ1}, the authors gave the existence of the conical K\"ahler-Ricci flow \eqref{CK} with initial metric $\tilde\omega=\omega+\dd\tilde\varphi$ whose potential $\tilde\varphi$ is continuous. Generalizing authors' arguments to the unnormalized conical K\"ahler-Ricci flow, we have the following result.
\begin{thm}[Theorem $1.2$ and Remark $1.3$ in \cite{JWLXZ19}]\label{0313000}
There exists a unique solution to the conical K\"ahler-Ricci flow 
 \begin{equation}\label{0313001}
\left\{
\begin{aligned}
 &\ \frac{\partial \omega_{\gamma}(t)}{\partial t}=-{\rm Ric}(\omega_{\gamma}(t))+(1-\gamma)[D]+\eta\\
 &\ \ \ \ \ \ \ \ \ \ \ \ \ \ \ \ \ \ \ \ \ \ \ \ \ \ \ \ \ \ \ \ \ \ \ \ \ \ \ \ \ \ \ \ \ \ \ \ \ \ \ \ \ , \\
 &\ \omega_{\gamma}(t)|_{t=0}=\tilde{\omega}
\end{aligned}
\right.
\end{equation}
on $(0,T^\eta_{\gamma,\max})\times X$ in the following sense
\begin{itemize}
  \item  For $0<\delta<T<T^\eta_{\gamma,\max}$, there exists a constant $C$ such that for any $t\in(\delta,T)$, 
\begin{equation}
C^{-1}\omega_\gamma\leqslant\omega_\gamma(t)\leqslant C\omega_\gamma\ \ \ \text{in}\ \ \ X\setminus D;
\end{equation}
   \item  In $(0,T^\eta_{\gamma,\max})\times(X\setminus D)$, $\omega_\gamma(t)$ satisfies the smooth twisted K\"ahler-Ricci flow
 \begin{equation}\label{0313002}
\frac{\partial \omega_{\gamma}(t)}{\partial t}=-{\rm Ric}(\omega_{\gamma}(t))+\eta;
\end{equation}
  \item On $(0,T^\eta_{\gamma,\max})\times X$, $\omega_\gamma(t)$ satisfies equation \eqref{0313001} in the sense of currents;
  \item There exists a metric potential $\varphi_\gamma(t)\in \mathcal{C}^{2,\alpha,\gamma}\left( X\right)\cap \mathcal{C}^{\infty}\left(X\setminus D\right)$ such that $\omega_\gamma(t)=\omega_{\gamma t}+\sqrt{-1}\partial\bar{\partial}\varphi_\gamma(t)$ and $\lim\limits_{t\rightarrow0^{+}}\|\varphi_\gamma(t)-\tilde\varphi\|_{L^{\infty}(M)}=0$;
  \end{itemize}
\end{thm}
\begin{defi}\label{0313007}
Let $u_0\in \mathcal{C}^0(X)$ be a $\omega$-psh function. By saying $u(t)\in \mathcal{C}^{2,\alpha,\gamma}\left( X\right)\cap \mathcal{C}^{\infty}\left(X\setminus D\right)$ is a subsolution to the equation
\begin{equation}\label{0313011}
\frac{\partial \varphi_{\gamma}(t)}{\partial t}=\log\frac{\left(\omega_{\gamma t}+\dd\varphi_{\gamma}(t)\right)^n}{\omega^n}+h_\gamma+(1-\gamma)\log|s|_h^2
\end{equation}
 with initial value $u_0$ on $(0,T)\times X$, we mean that $u(t)$ satisfies the following conditions.
\begin{itemize}
  \item  For $0<\delta<T'<T$, there exists a constant $C$ such that for any $t\in(\delta,T')$,
\begin{equation}
C^{-1}\omega_\gamma\leqslant\omega_{\gamma t}+\sqrt{-1}\partial\bar{\partial}u(t)\leqslant C\omega_\gamma\ \ \ \text{in}\ \ \ X\setminus D;
\end{equation}
   \item  In $(0,T)\times(X\setminus D)$, $u(t)$ satisfies equation
 \begin{equation}\label{0313008}
\frac{\partial u(t)}{\partial t}\leqslant\log\frac{\left(\omega_{\gamma t}+\dd u(t)\right)^n}{\omega^n}+h_\gamma+(1-\gamma)\log|s|_h^2;
\end{equation}
  \item $\lim\limits_{t\rightarrow0^{+}}\|u(t)-u_0\|_{L^{\infty}(M)}=0$.
  \end{itemize}
\end{defi}
\begin{defi}\label{0313009}
Let $v_0\in \mathcal{C}^0(X)$ be a $\omega$-psh function. By saying $v(t)\in \mathcal{C}^{2,\alpha,\gamma}\left( X\right)\cap \mathcal{C}^{\infty}\left(X\setminus D\right)$ is a supersolution to equation \eqref{0313011} with initial value $v_0$ on $(0,T)\times X$, we mean that $v(t)$ satisfies the following conditions.
\begin{itemize}
  \item  For $0<\delta<T'<T$, there exists a constant $C$ such that for any $t\in(\delta,T')$,
\begin{equation}
C^{-1}\omega_\gamma\leqslant\omega_{\gamma t}+\sqrt{-1}\partial\bar{\partial}v(t)\leqslant C\omega_\gamma\ \ \ \text{in}\ \ \ X\setminus D;
\end{equation}
   \item  In $(0,T)\times(X\setminus D)$, $v(t)$ satisfies equation
 \begin{equation}\label{0313010}
\frac{\partial v(t)}{\partial t}\geqslant\log\frac{\left(\omega_{\gamma t}+\dd v(t)\right)^n}{\omega^n}+h_\gamma+(1-\gamma)\log|s|_h^2;
\end{equation}
  \item $\lim\limits_{t\rightarrow0^{+}}\|v(t)-v_0\|_{L^{\infty}(M)}=0$.
  \end{itemize}
\end{defi}
\begin{defi}\label{03130090}
Let $\phi_0\in \mathcal{C}^0(X)$ be a $\omega$-psh function. By saying $\phi(t)\in \mathcal{C}^{2,\alpha,\gamma}\left( X\right)\cap \mathcal{C}^{\infty}\left(X\setminus D\right)$ is a solution to equation \eqref{0313011} with initial value $\phi_0$ on $(0,T)\times X$, we mean that $\phi(t)$ is both a subsolution and a supersolution.
\end{defi}
Next, we prove some maximum principles.
\begin{pro}\label{0313003}
If $u(t)$ is a subsolution to equation \eqref{0313011} with initial value $u_0\in \mathcal{C}^0(X)\cap {\rm PSH}(X,\omega)$, and $v(t)$ is a supersolution to equation \eqref{0313011} with initial value $v_0\in \mathcal{C}^0(X)\cap {\rm PSH}(X,\omega)$. Then 
\begin{equation}\label{0313006}
u(t)-v(t)\leqslant\sup\limits_X\left(u_0-v_0\right).
\end{equation}
\end{pro}
\begin{proof}
For any $0<t_1<T'<T$ and $a>0$. We consider $\Psi(t)=u(t)-v(t)+a\log|s|_h^2$ on $[t_1,T']\times X$. If we denote
\begin{equation}\label{0314001}
\hat{\Delta}=\int_0^1 g_{s}^{i\bar{j}}\frac{\partial^2}{\partial z^i\partial\bar{z}^j}ds,
\end{equation}
where $g_s$ is the metric corresponding to $\omega_s=\omega_{\gamma t}+s\dd u(t)+(1-s)\dd v(t)$. Then $\Psi(t)$ evolves along the following equation
\begin{equation}\label{0314002}
  \frac{\partial \Psi(t)}{\partial t}\leqslant\hat{\Delta}\Psi(t)-a\hat{\Delta}\log|s|_h^{2}.
\end{equation}
Since $\omega_{u(t)}$ and $\omega_{v(t)}$ are bounded from below by $c_0\omega_\gamma$ and $\omega_\gamma\geqslant c_1\omega$, we obtain
\begin{equation}\label{0314005}
\omega_s\geqslant c\omega
\end{equation}
for some constants $c_0$ and $c$ depending on $t_1$ and $T$.
Combining this with $-\sqrt{-1}\partial\bar{\partial}\log|s|_h^2=\theta$, we conclude that there exists a constant $C(t_1,T)$ depending on $t_1$ and $T$ such that
\begin{equation}\label{0314003}
-\hat{\Delta}\log|s|_h^2=\int_0^1\tr_{\omega_s}\theta ds\leq C(t_1,T)
\end{equation}
in $X\setminus D$. Then 
\begin{equation}\label{0314004}
  \frac{\partial \Psi(t)}{\partial t}\leqslant\hat{\Delta}\Psi(t)+aC(t_1,T).
\end{equation}

Let $\tilde{\Psi}=\Psi-aC(t_1,T)(t-t_1)-\epsilon (t-t_1)$. It is obvious that the space maximum of  $\Psi(t)$ on $[t_1,T']\times X$ is attained away from $D$.  Let $(t_0,x_0)$ be the maximum point. If $t_0>t_1$, by the maximum principle, at $(t_0,x_0)$, we have
\begin{equation}\label{0314006}
 0\leqslant \Big(\frac{\partial }{\partial t}-\hat{\Delta}\Big)\tilde{\Psi}(t)\leqslant -\epsilon,
\end{equation}
which is impossible, hence $t_0=t_1$. Then for $(t,x)\in [t_1,T]\times X$, we obtain
\begin{equation}\label{0314007}
u(t)-v(t)\leqslant \sup\limits_X\left(u(t_1)-v(t_1)\right)+aC(t_1,T)+\epsilon T-a\log|s|_h^2.
\end{equation}
Since $\lim\limits_{t\rightarrow0^+}\|u(t)-u_0\|_{L^{\infty}(M)}=0$ and $\lim\limits_{t\rightarrow0^{+}}\|v(t)-v_0\|_{L^{\infty}(M)}=0$, we let $a\rightarrow0$ and then $t_1\rightarrow0^+$, there holds
\begin{equation}\label{0314008}
u(t)-v(t)\leqslant \sup\limits_X\left(u_0-v_0\right)+\epsilon T.
\end{equation}
Then \eqref{0313006} is proved after we let $\epsilon\rightarrow0$.
\end{proof}
\begin{pro}\label{0318001}
Assume that $\phi(t)$ is a solution to equation \eqref{0313011} with initial value $\phi_0\in \mathcal{C}^0(X)\cap {\rm PSH}(X,\omega)$ and that $\psi(t)$ is a weak subsolution to equation \eqref{0313011} with initial value $\psi_0\in{\rm PSH}(X,\omega)$ admitting Lelong number zero. If $\psi_0\leqslant\phi_0$, then 
\begin{equation}\label{0318002}
\psi(t)-\phi(t)\leqslant\sup\limits_X\left(\psi_0-\phi_0\right).
\end{equation}
\end{pro}
\begin{proof}
For any $0<t_1<T'<T$, since $\psi(t)$ is a weak solution to equation \eqref{0313011}, there exists a constant $C$ such that
\begin{equation}\label{0318003}
C^{-1}\omega_\gamma\leqslant\omega_{\gamma t}+\sqrt{-1}\partial\bar{\partial}\psi(t)\leqslant C\omega_\gamma.
\end{equation}
Hence
\begin{equation}\label{0318004}
\frac{\partial \psi(t)}{\partial t}=\log\frac{\left(\omega_{\gamma t}+\dd\psi(t)\right)^n}{\omega^n}+h_\gamma+(1-\gamma)\log|s|_h^2
\end{equation}
is uniformly bounded in $[t_1,T']\times (X\setminus D)$, which implies that
\begin{equation}\label{0318005}
\|\psi(t)-\psi(t')\|_{L^\infty(X)}\leqslant C|t-t'|\ \ \ \text{for all}\ t,t'\in[t_1,T']
\end{equation}
and that $\psi(t_1)$ is continuous by Ko{\l}odziej's $L^p$-estimates, and so $\psi(t)$ is a solution to equation \eqref{0313011} with initial value $\psi(t_1)$. By using Proposition \ref{0313003}, on $[t_1,T']\times X$, we have
\begin{equation}\label{0418006}
\psi(t)-\phi(t)\leqslant\sup\limits_X\left(\psi(t_1)-\phi(t_1)\right).
\end{equation}
Since $\psi(t_1)$ converges to $\psi_0$ in $L^1$-sense and $\phi(t_1)$ converges to $\phi_0$ in $L^\infty$-sense, it then follows form Hartogs' Lemma that 
\begin{equation}\label{0418006}
\psi(t)-\phi(t)\leqslant\sup\limits_X\left(\psi_0-\phi_0\right)
\end{equation}
on $(0,T)\times X$ after letting $t_1\to0$.
\end{proof}
\begin{pro}\label{0314009}
Assume that $\psi(t)$ is a weak solution to equation \eqref{0313011} with initial value $\psi_0$ and that the Lelong number of $\psi_0$ is zero. Let $l$ be a constant such that $2l>T^\eta_{\gamma,\max}$, then there exists a uniform constant such that 
\begin{equation}\label{0314010}
(1-2lt)\psi_0-Ct+n(t\log t-t)\leqslant\psi(t)
\end{equation}
for all $t\in(0,\frac{1}{2l})$.
\end{pro}
\begin{proof}
Fix some $l\geqslant1$ such that $2l>T^\eta_{\gamma,\max}$ and $2l-1>\frac{1}{T^\eta_{\gamma,\max}}$. Taking a sufficiently small positive constant $t_1$ such that
\begin{equation}\label{0315001}
(2l-1)\omega+(1+2lt_1)\nu_\gamma>0.
\end{equation}
For any $t_1>0$, we consider equation
\begin{equation}\label{0314012}
(\omega+\dd\varphi_{\gamma t_1})^n=e^{\varphi_{\gamma t_1}-h_\gamma-2l\psi(t_1)}\frac{\omega^n}{\vert s\vert_h^{2(1-\gamma)}}.
\end{equation}
Since the Lelong numbers of $\psi_0$ and $\psi(t_1)$ are zero, and $\psi(t_1)$ converges to $\psi_0$ in $L^1$-sense as $t_1\to0$, from Lemma $3.5$ in \cite{NL2014} (see also Theorem $3.1$ in \cite{Zer01}), there holds 
\begin{equation}\label{0314013}
\int_Xe^{-ph_\gamma-2pl\psi(t_1)}\frac{\omega^n}{\vert s\vert_h^{2p(1-\gamma)}}\leqslant C
\end{equation}
for some uniform $p>1$ and constant $C$ independent of $t_1$. By Ko{\l}odziej's $L^p$-estimates, there exists a uniform constant $C$ independent of $t_1$ such that 
\begin{equation}\label{0314014}
\|\varphi_{\gamma t_1}\|_{C^0(X)}\leqslant C.
\end{equation}
We set $\Psi(t)=(1-2lt)\psi(t_1)+t\varphi_{\gamma t_1}+n(t\log t-t)$. From Guenancia-P$\breve{a}$un's result (Proposition $1$ in \cite{GP1}), we can conclude that $\varphi_{\gamma t_1}$ is $C^\alpha$ on $M$ and smooth in $X\setminus D$, and so is $\Psi(t)$. For $0\leqslant t\leqslant\frac{1}{2l}-t_1$,  we have 
\begin{equation}\label{0314015}
\begin{split}
\omega_{\gamma(t+t_1)}+\dd\Psi(t)&=\omega+(t+t_1)\nu_\gamma+(1-2lt)\dd \psi(t_1)+t\dd\varphi_{\gamma  t_1}\\
&=(1-2lt)\omega_{\psi(t_1)}+t\omega_{\varphi_{\gamma  t_1}}+t\left((2l-1)\omega+(1+2lt_1)\nu_\gamma\right)\\
&\geqslant t\omega_{\varphi_{\gamma  t_1}}.
\end{split}
\end{equation}
Hence
\begin{equation}\label{0315002}
\left(\omega_{\gamma(t+t_1)}+\dd\Psi(t)\right)^n\geqslant t^n\omega_{\varphi_{\gamma  t_1}}^n=t^ne^{\varphi_{\gamma t_1}-h_\gamma-2l\psi(t_1)}\frac{\omega^n}{\vert s\vert_h^{2(1-\gamma)}},
\end{equation}
which implies that
\begin{equation}\label{0315003}
\begin{split}
&\ \ \ \ \log\frac{(\omega_{\gamma(t+t_1)}+\dd\Psi(t))^n}{\omega^n}+h_\gamma+(1-\gamma)\log\vert s\vert_h^2\\
&\geqslant n\log t+\varphi_{\gamma t_1}-2l\psi(t_1)=\frac{\partial\Psi(t)}{\partial t}.
\end{split}
\end{equation}

We consider $\tilde\Psi(t)=\Psi(t)-\psi(t+t_1)+a\log|s|_h^2$ on $[t_1,\frac{1}{2l}-t_1]\times X$. If we denote
\begin{equation}\label{0315004}
\hat{\Delta}=\int_0^1 g_{s}^{i\bar{j}}\frac{\partial^2}{\partial z^i\partial\bar{z}^j}ds,
\end{equation}
where $g_s$ is the metric corresponding to $\omega_s=\omega_{\gamma t}+s\dd\Psi(t)+(1-s)\dd \psi(t+t_1)$. Then $\tilde\Psi(t)$ evolves along the following equation
\begin{equation}\label{0315005}
  \frac{\partial \tilde\Psi(t)}{\partial t}\leqslant\hat{\Delta}\tilde\Psi(t)-a\hat{\Delta}\log|s|_h^{2}.
\end{equation}
From equation \eqref{0314015}, for $t_1\leqslant t\leqslant\frac{1}{2l}-t_1$, we have
\begin{equation}\label{03150005}
\omega_{\gamma t}+\dd\Psi(t)\geqslant(1-2lt)\omega_{\psi(t_1)}\geqslant2lt_1\omega_{\psi(t_1)}\geqslant c(l,t_1)\omega,
\end{equation}
and so is
\begin{equation}\label{03150006}
\omega_s\geqslant c(l,t_1)\omega
\end{equation}
for some constants $c(l,t_1)$ depending on $l$ and $t_1$. Combining this with $-\sqrt{-1}\partial\bar{\partial}\log|s|_h^2=\theta$, we conclude that there exists a constant $C(l,\delta,t_1)$ such that
\begin{equation}\label{0315007}
-\hat{\Delta}\log|s|_h^2=\int_0^1\tr_{\omega_s}\theta ds\leq C(l,t_1)
\end{equation}
in $X\setminus D$. Then 
\begin{equation}\label{0315008}
  \frac{\partial \tilde\Psi(t)}{\partial t}\leqslant\hat{\Delta}\tilde\Psi(t)+aC(l,t_1).
\end{equation}

Let $\Phi(t)=\tilde\Psi-aC(l,t_1)t-\epsilon t$. It is obvious that the space maximum of  $\Phi(t)$ on $[0,\frac{1}{2l}-t_1]\times X$ is attained away from $D$.  Let $(t_0,x_0)$ be the maximum point. If $t_0>0$, by the maximum principle, at $(t_0,x_0)$, we have
\begin{equation}\label{0315009}
 0\leqslant \Big(\frac{\partial }{\partial t}-\hat{\Delta}\Big)\Phi(t)\leqslant -\epsilon,
\end{equation}
which is impossible, hence $t_0=0$. Then for $(t,x)\in [0,\frac{1}{2l}-t_1]\times X$, we obtain
\begin{equation}\label{0315010}
\Psi(t)-\psi(t+t_1)\leqslant \sup\limits_X\left(\Psi(0)-\psi(t_1)\right)+aTC(l,t_1)+\epsilon T-a\log|s|_h^2.
\end{equation}
Let $\epsilon\rightarrow0$ and then $a\rightarrow0$, there holds
\begin{equation}\label{0315011}
\Psi(t)-\psi(t+t_1)\leqslant \sup\limits_X\left(\Psi(0)-\psi(t_1)\right)=0.
\end{equation}
By using \eqref{0314014}, we have
\begin{equation}\label{0315012}
(1-2lt)\psi(t_1)-Ct+n(t\log t-t)\leqslant \psi(t+t_1).
\end{equation}
By using Remark \ref{0502001}, after letting $t_1\to0$, we have
\begin{equation}\label{0315013}
(1-2lt)\psi_0-Ct+n(t\log t-t)\leqslant \psi(t)
\end{equation}
for all $t\in(0,\frac{1}{2l})$.
\end{proof}
\begin{thm}\label{0318007}
Assume that $\psi(t)$ and $\phi(t)$ are weak solutions to equation \eqref{0313011} with initial values $\psi_0$ and $\phi_0$ respectively, where $\psi_0,\ \phi_0\in{\rm PSH}(X,\omega)$ with Lelong numbers zero. If $\psi_0\leqslant\phi_0$, then on $(0,T)\times X$, 
\begin{equation}\label{0318008}
\psi(t)\leqslant\phi(t).
\end{equation}
\end{thm}
\begin{proof}
We first prove this theorem under assumption that $\nu_\gamma$ is semi-positive. For any $t_1>0$, we consider 
\begin{equation}\label{0318008}
u(t)=\phi(t+t_1)+C(t_1)
\end{equation}
where $C(t_1)=Ct_1-n(t_1\log t_1-t_1)$ comes from \eqref{0314010}. Hence $u(t)$ is a solution to equation
\begin{equation}\label{0318009}
\frac{\partial u(t)}{\partial t}=\frac{\partial \phi(t+t_1)}{\partial t}=\log\frac{\left(\omega+(t+t_1)\nu_\gamma+\dd u(t)\right)^n}{\omega^n}+h_\gamma+(1-\gamma)\log|s|_h^2
\end{equation}
with initial value
\begin{equation}\label{0318009}
u(0)=\phi(t_1)+C(t_1)\geqslant\left(1-2lt_1\right)\phi_0.
\end{equation}
On the other hand, $\psi(t)$ satisfies equation
\begin{equation}\label{0318010}
\begin{split}
\frac{\partial \psi(t)}{\partial t}&=\log\frac{\left(\omega+t\nu_\gamma+\dd \psi(t)\right)^n}{\omega^n}+h_\gamma+(1-\gamma)\log|s|_h^2\\
&\leqslant\log\frac{\left(\omega+(t+t_1)\nu_\gamma+\dd \psi(t)\right)^n}{\omega^n}+h_\gamma+(1-\gamma)\log|s|_h^2,
\end{split}
\end{equation}
that is, $\psi(t)$ is a weak subsolution with initial value $\psi_0$. Then by using Proposition \ref{0318001}, we have
\begin{equation}\label{0318011}
\psi(t)-u(t)\leqslant\sup\limits_X\left(\psi_0-u(0)\right)\leqslant\sup\limits_X\left(\psi_0-\left(1-2lt_1\right)\phi_0\right)\leqslant\sup\limits_X\left(\psi_0-\phi_0\right)+2lt_1\sup\limits_X\phi_0.
\end{equation}
Let $t_1\to0$, there holds
\begin{equation}\label{0318012}
\psi(t)-\phi(t)\leqslant\sup\limits_X\left(\psi_0-\phi_0\right),
\end{equation}
that is,
\begin{equation}\label{0318020}
\psi(t)\leqslant\phi(t).
\end{equation}

For general $\nu_\gamma$, we consider 
\begin{equation}\label{0318016}
u(t)=\tilde{C}e^{t}\psi\left(\frac{1}{\tilde{C}}(1-e^{-t})\right).
\end{equation}
where $\tilde{C}$ is a constant satisfying $\tilde{C}>\frac{1}{T^\eta_{\gamma,\max}}$. Hence $u(t)$ is weak subsolution to equation
\begin{equation}\label{0318017}
\begin{split}
\frac{\partial u(t)}{\partial t}&=u(t)+\dot\psi\left(\frac{1}{\tilde{C}}(1-e^{-t})\right)\\
&=\log\frac{\left(\omega+\left(\frac{1}{\tilde{C}}(1-e^{-t})\right)\nu_\gamma+\dd \psi\left(\frac{1}{\tilde{C}}(1-e^{-t})\right)\right)^n}{\omega^n}+u(t)\\
&\ \ \ +h_\gamma+(1-\gamma)\log|s|_h^2\\
&=\log\frac{\left(\tilde{C}\omega+\left(e^{t}-1\right)\left(\tilde{C}\omega+\nu_\gamma\right)+\dd u(t)\right)^n}{\left(\tilde{C}\omega\right)^n}+v(t)-nt\\
&\ \ \ +h_\gamma+(1-\gamma)\log|s|_h^2\\
&\leqslant\log\frac{\left(\tilde{C}\omega+\left(e^{t+t_1}-1\right)\left(\tilde{C}\omega+\nu_\gamma\right)+\dd u(t)\right)^n}{\left(\tilde{C}\omega\right)^n}+v(t)-nt\\
&\ \ \ +h_\gamma+(1-\gamma)\log|s|_h^2
\end{split}
\end{equation}
with initial value $u(0)=\tilde{C}\psi_0$, where we use $\tilde{C}\omega+\nu_\gamma$ is positive in the last inequality since $\tilde{C}>\frac{1}{T^\eta_{\gamma,\max}}$.

We define
\begin{equation}\label{0318013}
v(t)=\tilde{C}e^{t+t_1}\phi\left(\frac{1}{\tilde{C}}(1-e^{-(t+t_1)})\right)+nt_1(e^t-1)+\tilde{C}e^{t+t_1}C\left(\frac{1}{\tilde{C}}\left(1-e^{-t_1}\right)\right).
\end{equation}
where $C\left(\frac{1}{\tilde{C}}\left(1-e^{-t_1}\right)\right)=C\cdot\frac{1}{\tilde{C}}\left(1-e^{-t_1}\right)-n\left(\left(\frac{1}{\tilde{C}}\left(1-e^{-t_1}\right)\right)\log \left(\frac{1}{\tilde{C}}\left(1-e^{-t_1}\right)\right)-\frac{1}{\tilde{C}}\left(1-e^{-t_1}\right)\right)$ comes from \eqref{0314010}. Hence $v(t)$ is a solution to equation
\begin{equation}\label{0318014}
\begin{split}
\frac{\partial v(t)}{\partial t}&=\tilde{C}e^{t+t_1}\phi\left(\frac{1}{\tilde{C}}(1-e^{-(t+t_1)})\right)+\dot\phi\left(\frac{1}{\tilde{C}}(1-e^{-(t+t_1)})\right)+nt_1e^t+\tilde{C}e^{t+t_1}C\left(\frac{1}{\tilde{C}}\left(1-e^{-t_1}\right)\right)\\
&=\log\frac{\left(\omega+\left(\frac{1}{\tilde{C}}(1-e^{-(t+t_1)})\right)\nu_\gamma+\dd \phi\left(\frac{1}{\tilde{C}}(1-e^{-(t+t_1)})\right)\right)^n}{\omega^n}+v(t)+nt_1\\
&\ \ \ +h_\gamma+(1-\gamma)\log|s|_h^2\\
&=\log\frac{\left(\tilde{C}\omega+\left(e^{t+t_1}-1\right)\left(\tilde{C}\omega+\nu_\gamma\right)+\dd v(t)\right)^n}{\left(\tilde{C}\omega\right)^n}+v(t)-nt\\
&\ \ \ +h_\gamma+(1-\gamma)\log|s|_h^2\\
\end{split}
\end{equation}
with initial value
\begin{equation}\label{0318015}
v(0)=\tilde{C}e^{t_1}\phi\left(\frac{1}{\tilde{C}}(1-e^{-t_1})\right)+\tilde{C}e^{t_1}C\left(\frac{1}{\tilde{C}}\left(1-e^{-t_1}\right)\right).
\end{equation}
Then by similar arguments as in the proof of Proposition \ref{0318001} and by using Proposition \ref{0314009}, we have
\begin{equation}\label{0318017}
u(t)-v(t)\leqslant\sup\limits_X\left(u(0)-v(0)\right)\leqslant\tilde{C}\sup\limits_X\left(\psi_0-e^{t_1}\left(1-2l\frac{1}{\tilde{C}}\left(1-e^{-t_1}\right)\right)\phi_0\right),
\end{equation}
which implies that
\begin{equation}\label{0318018}
\begin{split}
&\ \ \ \psi\left(\frac{1}{\tilde{C}}(1-e^{-t})\right)-e^{t_1}\phi\left(\frac{1}{\tilde{C}}(1-e^{-(t+t_1)})\right)\\
&\leqslant e^{-t}\sup\limits_X\left(\psi_0-\phi_0\right)+ e^{-t}\left(1-e^{t_1}\left(1-2l\frac{1}{\tilde{C}}\left(1-e^{-t_1}\right)\right)\right)\sup\limits_X\phi_0\\
&\ \ \ +\frac{nt_1}{\tilde{C}}(1-e^{-t})+e^{t_1}C\left(\frac{1}{\tilde{C}}\left(1-e^{-t_1}\right)\right).
\end{split}
\end{equation}
After letting $t_1\to0$, we have
\begin{equation}\label{0318019}
\psi\left(\frac{1}{\tilde{C}}(1-e^{-t})\right)-\phi\left(\frac{1}{\tilde{C}}(1-e^{-t})\right)\leqslant e^{-t}\sup\limits_X\left(\psi_0-\phi_0\right)\leqslant0,
\end{equation}
which implies that
\begin{equation}\label{0318021}
\psi(t)\leqslant\phi(t)\ \ \ \text{on}\ \ \ (0,1)\times X.
\end{equation}
Fix $t_0\in(0,1)$, from Remark \ref{0502001}, we know that $\psi(t)$ and $\phi(t)$ are solutions to equation \eqref{0313011} with initial values $\psi(t_0)$ and $\phi(t_0)$ respectively. Then by Proposition \ref{0313003}, we have
\begin{equation}\label{031802111}
\psi(t)\leqslant\phi(t)\ \ \ \text{on}\ \ \ (t_0,T)\times X.
\end{equation}
Combining \eqref{0318021} with \eqref{031802111}, we complete the proof.
\end{proof}
\begin{thm}\label{0318007}
Assume that the Lelong number of $\varphi_0\in{\rm PSH}(X,\omega)$ is zero. Then there exists a unique weak solution to equation \eqref{0313011} with initial value $\varphi_0$, that is, there exists a unique weak solution to the twisted conical K\"ahler-Ricci flow \eqref{CK} on $[0,T^\eta_{\gamma,\max})$. 
\end{thm}


\begin{thebibliography}{9}


\bibitem{CGP}
\textit{F.~Campana}, \textit{H.~Guenancia} and \textit{M.~P$\breve{a}$un},
Metrics with cone singularities along normal crossing divisors and holomorphic tensor fields, Annales scientifiques de l$^{'}$\'Ecole Normale Sup\'erieure, \textbf{46} (2013), 879--916.

\bibitem{Cao}
\textit{H.~D.~Cao},
Deformation of K\"ahler metrics to K\"ahler-Einstein metrics on compact K\"ahler manifolds, Inventiones Mathematicae, \textbf{81} (1985), 359--372.

\bibitem{CW}
\textit{X.~X.~Chen} and \textit{Y.~Q.~Wang},
Bessel functions, Heat kernel and the Conical K\"ahler-Ricci flow, Journal of Functional Analysis, \textbf{269} (2015), 551--632.

\bibitem{CW1}
\textit{X.~X.~Chen} and \textit{Y.~Q.~Wang},
 On the long-time behaviour of the Conical K\"ahler-Ricci flows, Journal f\"ur die reine und angewandte Mathematik, \textbf{744} (2018), 165--199.

\bibitem{TC}
\textit{T.~Collins} and \textit{G.~Sz$\acute{e}$kelyhidi},
 The twisted K\"ahler-Ricci flow, Journal f\"ur die Reine und Angewandte Mathematik, \textbf{716} (2016), 179--205.
 
 \bibitem{JPD}
\textit{J.~P.~Demailly},
Regularization of closed positive currents and intersection theory, Journal of Algebraic Geometry, \textbf{1} (1992), 361--409.
 
 \bibitem{NL2017}
\textit{E.~Di Nezza} and \textit{C.~H.~Lu},
Uniqueness and short time regularity of the weak K\" ahler-Ricci flow, Advances in Mathematics, \textbf{305} (2017), 953--993.

\bibitem{NL2014}
\textit{E.~Di Nezza} and \textit{C.~H.~Lu},
Complex Monge-Amp\`ere equations on quasi-projective varieties, Journal f\"ur die reine und angewandte Mathematik, \textbf{727} (2017), 145--167.
 
 \bibitem{SD2}
\textit{S.~K.~Donaldson},
K\"ahler metrics with cone singularities along a divisor. Essays in mathematics and its applications, Springer Berlin Heidelberg, (2012), 49--79.

\bibitem{GEDWA1}
\textit{G.~Edwards},
Metric contraction of the cone divisor by the conical K\"ahler-Ricci flow, Mathematische Annalen, \textbf{374} (2017), 1--33.

\bibitem{GEDWA}
\textit{G.~Edwards},
 A scalar curvature bound along the conical K\"ahler-Ricci flow, The Journal of Geometric Analysis, \textbf{28} (2018), 225--252.

\bibitem{EGZ}
\textit{P.~Eyssidieux}, \textit{V.~Guedj} and \textit{A.~Zeriahi},
A priori $L^\infty$-estimates for degenerate complex Monge-Amp\`ere equations, International Mathematics Research Notices, \textbf(2008), article ID rnn070.

 \bibitem{VGAZ2}
\textit{V.~Guedj} and \textit{A.~Zeriahi},
Regularizing properties of the twisted K\"ahler-Ricci flow, Journal f\"ur die Reine und Angewandte Mathematik, \textbf{729} (2017), 275--304.

\bibitem{GP1}
\textit{H.~Guenancia} and \textit{M.~P$\breve{a}$un},
Conic singularities metrics with perscribed Ricci curvature: the case of general cone angles along normal crossing divisors, Journal of Differential Geometry, \textbf{103} (2016), 15--57.

\bibitem{JMR}
\textit{T.~Jeffres}, \textit{R.~Mazzeo} and \textit{Y.~Rubinstein},
 K\"ahler-Einstein metrics with edge singularities, Annals of Mathematics, \text{183} (2016), 95-176.

\bibitem{K000}
\textit{S.~Ko{\l}odziej},
 The complex Monge-Amp\`ere equation, Acta mathematica,  \textbf{180} (1998), 69--117.

\bibitem{NVK}
\textit{N.~V.~Krylov},
Lectures on elliptic and parabolic equations in H\"older spaces, American Mathematical Society, (1996).

\bibitem{LSZ}
\textit{C.~Li}, \textit{L.~M.~Shen} and \textit{T. ~Zheng},
The K\"ahler-Ricci flow on log canonical pairs of general type, Journal of Functional Analysis, \textbf{285} (2023), 109984.

\bibitem{JWL}
\textit{J.~W.~Liu},
 The generalized K\"ahler Ricci flow, Journal of Mathematical Analysis and Applications,  \textbf{408} (2013), 751--761.

\bibitem{JWLCJZ}
\textit{J.~W.~Liu} and \textit{C.~J.~Zhang},
The conical complex Monge-Amp\`ere equations on K\"ahler manifolds, Calculus of Variations and Partial Differential Equations, \textbf{57} (2018), 44.

\bibitem{JWLXZ}
\textit{J.~W.~Liu} and \textit{X.~Zhang},
Conical K\" ahler-Ricci flow on Fano manifolds, Advances in Mathematics,  \textbf{307} (2017), 1324--1371.

\bibitem{JWLXZ1}
\textit{J.~W.~Liu} and \textit{X.~Zhang},
The conical K\"ahler-Ricci flow with weak initial data on Fano manifold, International Mathematics Research Notices, \textbf{17} (2017), 5343--5384.

\bibitem{JWLXZ19}
\textit{J.~W.~Liu} and \textit{X.~Zhang},
Cusp K\"ahler-Ricci flows on compact K\"ahler manifolds, Annali di Matematica Pura ed Applicata (1923-), \textbf{198} (2019), 289-306.

\bibitem{JWLXZ2}
\textit{J.~W.~Liu} and \textit{X.~Zhang},
Stability of  the conical K\"ahler-Ricci flows on Fano manifolds, Communications in Partial Differential Equations, \textbf{46} (2021), 953--1004.

\bibitem{JWLXZ3}
\textit{J.~W.~Liu} and \textit{X.~Zhang},
The openness theorems on the convergence of the twisted K\"ahler-Ricci flows, preprint, 2023. 

\bibitem{Nie}
\textit{X.~L.~Nie},
 Regularity of a complex Monge-Amp\`ere equation on Hermitian manifolds, Communications in Analysis and Geometry, \textbf{22} (2014), 833--856.

\bibitem{Nomura}
\textit{R.~Nomura},
Blow-up behavior of the scalar curvature along the conical K\"ahler-Ricci flow with finite time singularities, Differential Geometry and its Applications, \textbf{58} (2018), 1--16.

\bibitem{LMSH2}
\textit{L.~M.~Shen},
$C^{2,\alpha}$-estimate for conical K\"ahler-Ricci flow, Calculus of Variations and Partial Differential Equations, \textbf{57} (2018), 33.

\bibitem{LMSH1}
\textit{L.~M.~Shen},
Maximal time existence of unnormalized conical Kähler-Ricci flow, Journal f\"ur die reine und angewandte Mathematik, \textbf{760} (2020), 169--193.

\bibitem{Sko72}
\textit{H.~Skoda},
Sous-ensembles analytiques d'ordre fini ou infini dans $\mathbb{C}^n$, Bulletin de la Soci\'et\'e Math\'ematique de France, \textbf{100} (1972), 353-408.

\bibitem{JSGT}
\textit{J.~Song} and \textit{G.~Tian},
The K\"ahler-Ricci flow through singularities, Inventiones mathematicae, \textbf{207} (2017), 519--595.

\bibitem{GSVT}
\textit{G.~Sz\'ekelyhidi} and \textit{V.~Tosatti},
Regularity of weak solutions of a complex Monge-Amp\`ere equation, Analysis \& PDE,  \textbf{4} (2011), 369-378.

\bibitem{Tianzzhang}
\textit{G.~Tian} and \textit{Z.~Zhang}, On the K\"ahler-Ricci flow on projective manifolds of general type, Chinese Annals of Mathematics, Series B, \textbf{27} (2006), 179--192.

\bibitem{TaTo1}
\textit{T.~D.~T$\hat{o}$},
Regularizing properties of complex Monge-Amp\`ere flows, Journal of Functional Analysis, \textbf{272} (2017), 2058--2091.

\bibitem{TaTo2}
\textit{T.~D.~T$\hat{o}$},
Regularizing properties of complex Monge-Amp\`ere flows II, Mathematische Annalen, \textbf{372} (2018), 699--741.

\bibitem{Tsuji}
\textit{H.~Tsuji}, Existence and degeneration of K\"ahler-Einstein metrics on minimal algebraic varieties of general type, Mathematische Annalen, \textbf{281} (1988), 123--133.

\bibitem{YQW}
\textit{Y.~Q.~Wang},
Smooth approximations of the conical K\"ahler-Ricci flows, Mathematische Annalen, \textbf{365} (2016), 835--856.

\bibitem{Zer01}
\textit{A.~Zeriahi},
Volume and capacity of sublevel sets of a Lelong class of plurisubharmonic functions, Indiana University Mathematics Journal, \textbf{50} (2001), 671--703.

\bibitem{KWZ}
\textit{K.~W.~Zhang},
Some refinements of the partial $C^0$ estimate, Analysis \& PDE, \textbf{14} (2021), 2307--2326.

\bibitem{YSZ1}
\textit{Y.~S.~Zhang},
On a twisted conical K\"ahler-Ricci flow, Annals of Global Analysis and Geometry, \textbf{55} (2019), 69--98.

\bibitem{YSZ}
\textit{Y.~S.~Zhang},
A note on conical K\"ahler-Ricci flow on minimal elliptic K\"ahler surfaces, Acta Mathematica Scientia, \textbf{38} (2018), 169--176.

\end{thebibliography}
\end{document}